\documentclass[final]{siamltex}

\usepackage{amsfonts}
\usepackage{graphicx}
\usepackage{subfig}
\usepackage{booktabs}
\usepackage{enumerate}

\usepackage{amsmath}
\usepackage{amssymb}
\usepackage{epsf}

\usepackage{multirow}
\usepackage{makecell}

\usepackage{textcomp}

\usepackage{morefloats}

\def\complex{ {\mathbb C}}

\newtheorem{cor}{Corollary}
\newtheorem{lem}{Lemma}

\newtheorem{remark}{Remark}

\DeclareFontFamily{OT1}{pzc}{}
\DeclareFontShape{OT1}{pzc}{m}{it}{<-> s * [1.10] pzcmi7t}{}
\DeclareMathAlphabet{\mathpzc}{OT1}{pzc}{m}{it}

\usepackage{xcolor}

\begin{document}

\thispagestyle{empty}
\bibliographystyle{siam}

\title{On the limiting behavior of parameter-dependent network centrality measures}

\author {
Michele Benzi\thanks{Department of Mathematics and Computer
Science, Emory University, Atlanta, Georgia 30322, USA
(benzi@mathcs.emory.edu). The work of this author was supported
by National Science Foundation grants
DMS1115692 and DMS-1418889.} \and
Christine Klymko\thanks{Center
for Applied Scientific Computing, Lawrence Livermore
National Laboratory, Livermore, CA 94550, USA 
(klymko1@llnl.gov). This work was performed under the auspices of the 
U.S.~Department of
Energy by Lawrence Livermore National Laboratory under Contract
DE-AC52-07NA27344.
}
}

\maketitle

\markboth{{\sc Michele Benzi and Christine Klymko}}
{Parameter-dependent centrality measures}

\begin{abstract}
We consider a broad class of walk-based, parameterized node centrality measures 
for network analysis. These measures are expressed in terms of functions
of the adjacency matrix and generalize various
well-known centrality indices, including Katz and subgraph centrality. 
We show that the parameter can be ``tuned"
to interpolate between degree and eigenvector centrality, which
appear as limiting cases. Our analysis helps explain certain
correlations often observed between the rankings obtained using
different centrality measures, and provides some guidance
for the tuning of parameters.  We also highlight the 
roles played by the spectral gap of the adjacency matrix and
by the number of triangles in the network.
Our analysis covers both undirected and directed networks, 
including weighted ones. A brief discussion of PageRank is also given.
\end{abstract}



\section{Introduction}
The mathematical and computational study of complex networks
has experienced tremendous growth in recent years. A wide
variety of highly interconnected systems, both in nature
and in the man-made world of technology, can be modeled
in terms of networks. 
Network models
are now commonplace not only in the ``hard" sciences but
also in economics, finance, anthropology, urban studies, and even in the
humanities. As more and more data has become available, 
the need for tools to
analyze these networks has increased and a new field of
Network Science has come of age 
\cite{Linked,CH10,Ebook11,New10}.

Since graphs, which are abstract 
models of real-world networks, can be described in terms
of matrices, it comes as no surprise that linear algebra
plays an important role in network analysis. Many problems in
this area require the solution of linear systems, the
computation of eigenvalues and eigenvectors, and the evaluation
of matrix functions. Also, the study of dynamical processes on graphs
gives rise to systems of differential and 
difference equations posed on graphs \cite{BBV08}; the behavior of the
solution as a function of time is strongly influenced by the
structure (topology) of the underlying graph, which in turn
is reflected in the spectral properties of matrices associated
with the graph \cite{Bocca06}.

One of the most basic questions 
about network structure is the identification of
the ``important'' nodes in a network.  Examples include essential
proteins in Protein-Protein Interaction Networks, keystone species in
ecological networks, authoritative web pages 
on the World Wide Web, influential
authors in scientific collaboration networks, leading actors in
the Internet Movie Database, and so forth; see, e.g.,
\cite{Ebook11,New03} for details and many additional examples.
When the network 
being examined is very small (say, on the order of 10 nodes), this 
determination of importance can often be done visually, but as networks 
increase in size and complexity, visual analysis becomes impossible.  
Instead, computational 
measures of node importance, called {\em centrality measures}, 
are used to rank the nodes in a network. There are many 
different centrality measures in use; see, for example,
\cite{Brandes,Ebook11,Whos,New10} for extensive treatments of centrality and
discussion of different ranking methods. Many authors, however,
have noted that different centrality measures often provide 
rankings that are highly correlated, at least when
attention is restricted to the most highly ranked nodes; see. e.g., 
\cite{Dingetal,FR,Lee,MiPapa}, as well as 
the results in \cite{BKlym13}.  

In this paper we analyze the relationship between degree centrality, 
eigenvector centrality, and various centrality measures based on the 
diagonal entries (for undirected graphs) and row sums of certain
(analytic) functions of the adjacency matrix of the graph. 
These measures contain as special cases the well-known Katz centrality,
subgraph centrality, total communicability, and other centrality
measures which depend on a tuneable parameter.
We also include a brief discussion of PageRank \cite{Page}.
We point out that Kleinberg's HITS algorithm \cite{HITS},
as a type of eigenvector centrality, is covered by our analysis, as
is the extension of subgraph centrality to digraphs given in
\cite{BEK13}. 

As mentioned, there are a number of other ranking methods 
in use,
yet in this paper we limit ourselves to considering 
centrality measures based on functions of the adjacency
matrix, in addition to degree and eigenvector centrality. 
The choice of which of the many centrality 
measures to study and why is something that must be considered 
carefully; see the discussion in \cite{ChKrLaPe11}.  
In this paper we focus our attention
on centrality measures  that have been widely tested and 
that can be expressed in terms of linear algebra (more
specifically, in terms of the adjacency matrix 
of the network). We additionally 
restrict our scope to centrality measures that we can demonstrate 
(mathematically) to be
related to one other. Hence, we did not include
in our analysis two popular centrality measures,
betweenness centrality \cite{Fr77} and closeness
centrality \cite{Fr79}, which do not appear to
admit a simple expression in terms of the
adjacency matrix. 
Our results help explain the correlations often observed
between the rankings produced by different centrality 
measures, and may be useful in tuning parameters when
performing centrality calculations.

The paper is organized as follows. Sections 2 and 3 contain background
information on graphs and centrality measures.
In section 4 we describe the general class of functional centrality
measures considered in this paper and present some technical lemmas
on power series needed for our analysis. In section 5 we state and
prove our main results, which show that degree and eigenvector
centrality are limiting cases of the parameterized ones. 
Section 6 contains a brief discussion of the limiting behavior
of PageRank and related techniques.
In section 7 we provide an interpretation of our results in terms
of graph walks and discuss 
the role played by the spectral 
gap and by triangles in the network. Related work is
briefly reviewed in section 8. A short summary of numerical 
experiments aimed at illustrating the theory is given in
section 9 (the details of the experiments can be found in
the Supplementary Materials accompanying this paper).
Conclusions are given in section 10.

\section{Background and definitions} 
In this section we recall some basic concepts from graph theory  
that will be used in the rest of the paper.
A more complete overview can be found, e.g., in \cite{Die00}.
For ease of exposition only {\em unweighted} and {\em loopless}
graphs are considered in this section, but nearly all of our results admit a
straightforward generalization to graphs with (positive)
edge weights, and several of the results also apply
in the presence of loops; see the end of section 5, as well 
as section 6.

A {\em directed graph}, or {\em digraph}, $G=(V,E)$ is defined by a set of $n$ nodes 
(also referred to as vertices) $V$ and a set of edges 
$E = \{(i,j)\, |\, i,j \in V\}$. Note that, in general, $(i,j)\in E$
does not imply $(j,i)\in E$. When this happens, $G$ is {\em undirected} 
and the edges 
are formed by unordered pairs of vertices.  The {\em out-degree} of a 
vertex $i$, denoted by $d_i^{out}$, is given by the number of edges 
with $i$ as the starting node, i.e., the number of edges in $E$
of the form $(i,k)$. Similarly, the {\em in-degree} of node $i$
is the number $d_i^{in}$ of edges of the form $(k,i)$. If $G$ is undirected
then $d_i^{out} = d_i^{in} = d_i$, the {\em degree} of node $i$.

A {\em walk} of length $k$ in $G$ is a list of nodes $i_1, i_2, 
\ldots i_k, i_{k+1}$ such that for all $1 \leq l \leq k$, there is 
a (directed) edge between $i_l$ and $i_{l+1}$.  A {\em closed walk} is a walk 
where $i_1 = i_{k+1}$.  A {\em path} is a walk with no repeated nodes, and
a {\em cycle} is a closed walk  with no repeated nodes
except for the first and the last one.
A graph is {\em simple} if it has no {\em loops} (edges 
from a node $i$ to itself), no multiple edges, and unweighted edges.  
An undirected graph is {\em connected} if there exists a path between 
every pair of nodes. A directed graph is {\em strongly connected}
if there exists a directed path between every pair of nodes. 

Every graph $G$ can be represented as a matrix through the use of 
an {\em adjacency matrix} $A = (a_{ij})$ with 
$$a_{ij}=\left\{\begin{array}{ll}
1,& \textnormal{ if } (i,j) \textnormal{ is an edge in } G,\\
0, & \textnormal{ else. }
\end{array}\right .
$$

If $G$ is a simple, undirected graph, $A$ is binary and symmetric 
with zeros along the main diagonal.  In this case, the eigenvalues 
of $A$ will be real.  We label the eigenvalues of $A$ in non-increasing 
order: $\lambda_1 \geq \lambda_2 \geq \ldots \geq \lambda_n.$  If 
$G$ is connected, then $\lambda_1 > \lambda_2$ by the Perron-Frobenius 
theorem \cite[page 673]{Meyer00}.   
Since $A$ is a symmetric, real-valued matrix, we can 
decompose $A$ into $A=Q\Lambda Q^T$ where $\Lambda = {\rm diag}(\lambda_1, 
\lambda_2, \ldots, \lambda_n)$ with $\lambda_1 > \lambda_2 \geq 
\ldots \geq \lambda_n$, $Q=[{\bf q}_1, {\bf q}_2, \ldots, {\bf q}_n]$ 
is orthogonal, and ${\bf q}_i$ is the eigenvector associated with $\lambda_i$.
The {\em dominant eigenvector}, ${\bf q}_1$, can be chosen to have positive
entries when $G$ is connected: we write this ${\bf q}_1 > {\bf 0}$.

If $G$ is a strongly connected digraph, its adjacency matrix $A$ is 
irreducible, and conversely. Let $\rho(A) = r$ be the spectral radius of $A$.  Then, 
again
by the Perron-Frobenius theorem, $\lambda_1 = r$ is a simple eigenvalue 
of $A$ and both the left and right eigenvectors of $A$ associated 
with $\lambda_1$ can be chosen to be positive. If $G$ is also diagonalizable, then 
there exists an invertible matrix $X$ such that $A=X\Lambda X^{-1}$ 
where $\Lambda = {\rm diag}(\lambda_1, \lambda_2, \ldots, \lambda_n)$ 
with $\lambda_1\geq |\lambda_i|$ for $2 \leq i \leq n$, $X=[{\bf x}_1, 
{\bf x}_2, \ldots, {\bf x}_n]$, and $(X^{-1})^T=[{\bf y}_1, 
{\bf y}_2, \ldots, {\bf y}_n]$.  The left eigenvector associated with 
$\lambda_i$ is ${\bf y}_i$  and the right eigenvector associated with 
$\lambda_i$ is ${\bf x}_i$.  In the case where $G$ is not diagonalizable, 
$A$ can be decomposed using the Jordan canonical form:
\begin{displaymath}\label{Dpath}
A = XJX^{-1}= X \left(\begin{array}{cc}
\lambda_1 & {\bf 0}\\
{\bf 0} & \hat{J} \\  \end{array}\right) X^{-1} \,,
\end{displaymath}
where $J$ is the Jordan matrix of $A$, except that we place the 
$1 \times 1$ block corresponding to $\lambda_1$ first for
notational convenience.
The first column ${\bf x}_1$ of $X$ is the dominant right eigenvector of $A$
 and the first column ${\bf y}_1$ of 
$X^{-T}$ is the dominant left eigenvector of $A$ (equivalently, the dominant right
eigenvector of $A^T$).

Throughout the paper, $I$ denotes the $n\times n$ identity matrix.

\section{Node centrality}\label{sec:measures}
As we discussed in the Introduction,
many measures of node centrality have been developed and used over 
the years.  
In this section we review and motivate several centrality measures
to be analyzed in the rest of the paper.

\subsection{Some common centrality measures}
Some of the most common measures include degree centrality, 
eigenvector centrality \cite{Bonacich}, 
PageRank \cite{Page},
betweenness centrality \cite{Br08,Fr77}, Katz centrality \cite{Katz}, 
and subgraph centrality \cite{NetworkProp,estradarodriguez05}.   
More recently, total communicability has been introduced as a 
centrality measure \cite{BKlym13}. 
A node is deemed ``important" according to a given centrality
index if the corresponding value
of the index is high relative to that of other nodes.

Most of these measures are applicable
to both undirected and directed graphs. In the directed case,
however, each node can play two roles: {\em sink} and {\em source}, or
{\em receiver} and {\em broadcaster}, since a node in general can
be both a starting point and an arrival point for directed edges.   
This has led to the notion of {\em hubs} and {\em authorities}
in a network, with hubs being nodes with a high broadcast centrality
index and authorities being nodes with a high receive
centrality index.  For the types of indices considered in this
paper, broadcast centrality measures correspond to quantities
computed from the adjacency matrix $A$, whereas authority
centrality measures correspond to the same quantities computed
from the transpose $A^T$ of the adjacency matrix. 
When the graph is undirected, $A=A^T$ and the broadcast
and receive centrality scores of each node coincide.

Examples of (broadcast) centrality measures are:
\begin{itemize}
\item out-degree centrality: $d_i^{out}:= [A{\bf 1}]_i$;
\item (right) eigenvector centrality: $ C_{ev}(i) := {\bf e}_i^T {\bf q}_1=q_{1}(i)$,
where ${\bf q}_1$ is the dominant (right) eigenvector of $A$;
\item exponential subgraph centrality: $SC_i(\beta)
 := [{\rm e}^{\beta A}]_{ii}$;
\item resolvent subgraph centrality: $RC_i(\alpha) := 
[(I-\alpha A)^{-1}]_{ii}$,
\item total communicability: 
$TC_i(\beta) := [{\rm e}^{\beta A}{\bf 1}]_i = {\bf e}_i^T 
{\rm e}^{\beta A}{\bf 1}$;
\item Katz centrality: $K_i(\alpha) := [(I-\alpha A)^{-1} {\bf 1}]_{i} 
= {\bf e}_i^T (I-\alpha A)^{-1} {\bf 1}$.
\end{itemize}

Here ${\bf e}_i$ is the $i$th standard basis vector,
${\bf 1}$ is the vector of all ones, 
$0 < \alpha < \frac{1}{\lambda_1}$ (see below), 
and $\beta >0$. We note that the vector of all ones is
sometimes replaced by a {\em preference vector} $\bf v$
with positive entries; for instance, ${\bf v} = {\bf d} := A{\bf 1}$
(the vector of node out-degrees). 

Replacing $A$ with $A^T$ in the definitions above we obtain the
corresponding authority measures. Thus, out-degree centrality
becomes in-degree centrality, right eigenvector centrality 
becomes left eigenvector centrality, and row sums are
replaced by column sums when computing the total communicability
centrality. Note, however,
that the exponential and resolvent subgraph centralities are
unchanged when replacing $A$ with $A^T$, since $f(A^T) = f(A)^T$
for any matrix function \cite[Theorem 1.13]{High}. Hence, measures
based on the diagonal entries
cannot differentiate between the two roles a node can
play in a directed network, and for this reason they are mostly
used in the undirected case only (but see \cite{BEK13} for an
adaptation of subgraph centrality to digraphs). 

Often, the value $\beta =1$ is used in the calculation of 
exponential subgraph 
centrality and total communicability. The parameter
$\beta$ can be interpreted 
as an {\em inverse temperature} and has been used to model the effects 
of external disturbances on the network. As $\beta \to 0+$, 
the ``temperature" of the environment surrounding the network 
increases, corresponding to more intense external disturbances.  
Conversely, as $\beta \to \infty$, the temperature goes to 0 and the 
network ``freezes.'' We refer the reader to \cite{EHB11} for
an extensive discussion and applications of these physical analogies. 

\subsection{Justification in terms of graph walks}\label{subsec:walks}
The justification behind using the (scaled) matrix 
exponential to compute centrality measures can be seen by considering the 
power series expansion of ${\rm e}^{\beta A}$:
\begin{equation}\label{SC}
	{\rm e}^{\beta A} = I + \beta A + \frac{(\beta A)^2}{2!}+ \cdots + \frac{(\beta A)^k}{k!} + \cdots = \sum_{k=0}^{\infty} \frac{(\beta A)^k}{k!}.
\end{equation}

It is well-known that given an adjacency matrix 
$A$ of an unweighted network, $[A^k]_{ij}$ counts the total number 
of walks of length $k$ between nodes $i$ and $j$.  Thus 
$SC_i(\beta)=[{\rm e}^{\beta A}]_{ii}$, the exponential 
subgraph centrality of node $i$, counts the total number of closed 
walks in the network which are centered at node $i$, weighing walks 
of length $k$ by a factor of $\frac{\beta^k}{k!}$. 
Unlike degree, which is a purely local index, subgraph
centrality takes into account the short, medium and long range influence
of all nodes on a given node (assuming $G$ is strongly connected). 
Assigning decreasing weights to longer
walks ensures the convergence of the series (\ref{SC}) while
guaranteeing that short-range interactions are given more weight
than long-range ones \cite{estradarodriguez05}.

Total communicability is closely related to 
subgraph centrality.  This measure also counts the number of walks 
starting at node $i$, scaling walks of length $k$ by $\frac{\beta^k}{k!}$.  
However, rather than just counting closed walks, total 
communicability counts all walks between node $i$ and every node in the 
network. The name stems from the fact that $TC_i(\beta)=\sum_{j=1}^n
C_{ij}(\beta)$ where $C_{ij}(\beta):=[{\rm e}^{\beta A}]_{ij}$, the
communicability between nodes $i$ and $j$, is a measure of how
``easy" it is to exchange a message between nodes $i$ and $j$ over
the network;
see \cite{EHB11} for details.  
  Although subgraph centrality and total communicability 
are clearly related, they do not always provide the same 
ranking of the nodes.  Furthermore, unlike subgraph centrality,
total communicability can distinguish between the two roles
a node can play in a directed network. 
More information about the relation between the two measures can be 
found in \cite{BKlym13}.

The matrix resolvent $(I-\alpha A)^{-1}$ was first used to rank nodes 
in a network in the early 1950s, when Katz  used the column sums  
to calculate node importance \cite{Katz}. Since then, the diagonal 
values have also been used as a centrality measure, see 
\cite{NetworkProp}.  The resolvent subgraph centrality score of node 
$i$ is given by $[(I-\alpha A)^{-1}]_{ii}$ and the Katz centrality 
score is given by either $[(I-\alpha A)^{-1} {\bf 1}]_{i}$ or
$[(I-\alpha A^T)^{-1} {\bf 1}]_{i}$, depending on whether hub or
authority scores are desired. 
As mentioned, $\bf 1$ may be replaced by an arbitary (positive)
preference vector, $\bf v$.

As when using the matrix exponential, these resolvent-based centrality 
measures count the number of walks in the network, penalizing longer 
walks.  This can be seen by considering the Neumann series expansion 
of $(I-\alpha A)^{-1}$, valid for $0 < \alpha < \frac{1}{\lambda_1}$:
\begin{equation}
(I-\alpha A)^{-1} = I + \alpha A + \alpha^2 A^2 +\cdots + \alpha^k A^k 
+ \cdots = \sum_{k=0}^{\infty} \alpha^kA^k.
\label{eq:powerseries}
\end{equation}
The resolvent subgraph centrality of node $i$,  $[(I-\alpha A)^{-1}]_{ii}$, 
counts the total number of closed walks in the network which are centered 
at node $i$, weighing walks of length $k$ by $\alpha^k$.  Similarly, 
the Katz centrality 
of node $i$ counts all walks beginning at node $i$, penalizing the 
contribution of walks of length $k$ by $\alpha^k$.   The bounds on 
$\alpha$ ($0 < \alpha < \frac{1}{\lambda_1}$) ensure that the matrix 
$I-\alpha A$ is invertible and that the power series in 
(\ref{eq:powerseries}) converges to its inverse.  The bounds on 
$\alpha$ also force $(I-\alpha A)^{-1}$ to be nonnegative, as 
$I- \alpha A$ is a nonsingular $M$-matrix. 
Hence, both the diagonal entries and the row/column sums of
$(I-\alpha A)^{-1}$ are positive and can thus be used
for ranking purposes.

\section{A general class of functional centrality measures}
In this section we establish precise conditions that a matrix
function $f(A)$, where $A$ is the adjacency matrix of a network,
should satisfy in order to be used in the definition of walk-based 
centrality measures. We consider in particular analytic functions
expressed by power series, with a focus on issues like convergence, 
positivity, and dependence on a tuneable parameter $t>0$. 
We also formulate some auxiliary results on power series that
will be crucial for the analysis to follow.
For an introduction to the properties of analytic functions
see, e.g., \cite{MH87}. 

\subsection{Admissible matrix functions}
As discussed in subsection \ref{subsec:walks}, walk-based 
centrality measures (such as Katz or subgraph
centrality) lead to power series expansions in the (scaled) adjacency matrix 
of the network. While exponential- and resolvent-based centrality measures
are especially
natural (and well-studied), there are {\em a priori} infinitely many other 
matrix functions which could be used \cite{REG07,NetworkProp}. 
Not every function of the adjacency matrix, however, is suitable
for the purpose of defining centrality measures, and some restrictions
must be imposed. 

A first obvious condition is that the function should be
defined by a power series with {\em real} coefficients. This guarantees that
$f(z)$ takes real values when the argument is real, and that $f(A)$ has real
entries for any real $A$.
In \cite{REG07} (see also \cite{NetworkProp}), the authors proposed to consider
only analytic functions admitting a Maclaurin series expansion of the form
\begin{equation}\label{F1}
f(z) = \sum_{k=0}^{\infty} c_k z^k, \quad c_k \ge 0 \quad{\rm for } \quad
 k\ge 0\,.
\end{equation}

This ensures that $f(A)$ will be nonnegative for any adjacency matrix $A$.
In \cite{NetworkProp} it is further required that 
$c_k > 0$ for all $k = 1,2,\ldots, n-1$, so as to guarantee that $[f(A)]_{ij}>0$ 
for all $i\ne j$ whenever the network is (strongly) connected.\footnote{We
recall that a nonnegative $n\times n$ matrix $A$ is irreducible if and
only if $(I+A)^{n-1} >0$. See, e.g., \cite[Theorem 6.2.24]{HJ}.}   
Although not explicitly stated in \cite{NetworkProp}, it is clear that 
if one wants {\em all} the walks (of any length) in $G$ to make a 
positive contribution to a centrality measure based on $f$, then one
should impose the more restrictive condition $c_k > 0$ for all $k\ge 0$. 
Note that $c_0$
plays no significant role, since it's just a constant value added to all the
diagonal entries of $f(A)$ and therefore does not affect the rankings. However,
imposing $c_0 > 0$ guarantees that all entries of $f(A)$ are positive, and leads
to simpler formulas.
Another tacit assumption in \cite{NetworkProp} is that only power
series with a positive radius of convergence should be considered.

In the following, we will denote by $\cal P$ the class of analytic functions
that can be expressed as sums of power series with strictly positive
coefficients on some open neighborhood of $0$.
We note in passing that $\cal P$ forms a {\em positive cone} in function
space, i.e., $\cal P$ is closed under linear combinations with positive
coefficients.

Clearly, given an arbitrary adjacency matrix $A$, the matrix function
$f(A)$, with $f\in {\cal P}$, need not be defined; ndeed, $f$ must be
defined on the spectrum of $A$ \cite{High}. If $f$ is entire (i.e.,
analytic in the whole complex plane, like
the exponential function) then $f(A)$ will always be defined, but this
is not the case of functions with singularities, such as the resolvent. 
However, this difficulty can be easily circumvented by introducing
a (scaling) parameter $t$, and by considering for a given $A$ the parameterized
matrix function $g(t,A):= f(tA)$ only for values of $t$ such that
the power series 
$$f(tA) = c_0 I + c_1 t A + c_2 t^2 A^2 +\cdots = \sum_{k=0}^\infty c_k t^k A^k$$
is convergent; that is, such that $|t \lambda_1(A)|< R_f$, where $R_f$ denotes the radius
of convergence of the power series representing $f$. In practice, for the purposes
of this paper, we will limit ourselves to {\em positive} values of $t$
in order to guarantee that $f(tA)$ is entry-wise positive, as required
by the definition of a centrality index.
We summarize our discussion so far in the following lemma.

\vskip 0.05in

\begin{lem}\label{Lemma1}
Let $\cal P$ be the class of all analytic functions that can be expressed by a
Maclaurin series with strictly positive coefficients in an open disk centered at $0$.
Given an irreducible adjacency matrix $A$ and a function $f\in \cal P$ with radius
of convergence $R_f>0$, let $t^* = R_f/\lambda_1(A)$. Then $f(tA)$
is defined and strictly positive for all $t\in (0,t^*)$. If $f$ is entire,
then one can take $t^* = \infty$.
\end{lem}

\vskip 0.05in

Restriction of $f$ to the class $\cal P$ and use of a positive
parameter $t$, which will depend on $A$ and $f$ in case $f$ is not entire, 
allows one to define the notion of $f$-centrality (as well as $f$-communicability,
$f$-betweenness, and so forth, see \cite{NetworkProp}). 
Exponential subgraph centrality (with $t=\beta$) is an example
of an entire function (hence all positive values of $\beta$ are feasible),
while resolvent subgraph centrality
(with $t = \alpha$) exemplifies the situation where the parameter 
must be restricted to a finite interval, in this case $(0,\frac{1}{\lambda_1(A)})$
(since the geometric series $1 + x + x^2 + \cdots$
has radius of convergence $R_f = 1$).

We consider now two subclasses of the class $\cal P$ previously introduced.
We let ${\cal P}_\infty$ denote the set of all power series in $\cal P$ with radius of
convergence $R_f =\infty$, and with ${\cal P}^\infty$ the set of all power
series with finite radius of convergence $R_f$ such that 
\begin{equation}\label{cond}
\sum_{k=0}^\infty c_k R_f^k = \lim_{t\to 1^-} \sum_{k=0}^\infty c_k t^k R_f^k = \infty
\end{equation}
(we note that the first equality above follows from Abel's Theorem \cite[p.~229]{MH87}).
The exponential and the resolvent are representative
of functions in ${\cal P}_\infty$ and ${\cal P}^\infty$, respectively.
It is worth emphasizing that together, ${\cal P}_\infty$ and ${\cal P}^\infty$
do not exhaust the class $\cal P$. For example, the function
$f(z)=\sum_{k=0}^\infty \frac{z^k}{k^2}$  is in $\cal P$, but it is not in
${\cal P}_\infty$ (since its radius of convergence is $R_f=1$) or in ${\cal P}^\infty$, since
$$\lim_{t\to 1^-} \sum_{k=0}^\infty \frac{t^kR_f^k}{k^2} = \sum_{k=0}^\infty \frac{1}{k^2}
=\frac{\pi^2}{6} < \infty.$$     

In section \ref{sec:undirected_cent} we will analyze centrality measures based on
functions $f$ in ${\cal P}$ and its subclasses,
${\cal P}_\infty$ and ${\cal P}^\infty$.

\subsection{Asymptotic behavior of the ratio of two power series} 
In our study of the limiting behavior of parameter-dependent functional
centrality measures we will need to investigate the asymptotic behavior
of the ratio of two power series with positive coefficients. The following 
technical lemmas will
be crucial for our analysis. 

\vskip 0.05in

\begin{lem}\label{Lemma2}
Let the power series $\sum_{k=0}^\infty a_kt^k$, $\sum_{k=0}^\infty b_kt^k$ 
have positive real coefficients and be
convergent for all $t\ge 0$.  
If ${\displaystyle {\lim_{k\to \infty} \frac{a_k}{b_k} = 0}}$,
then
$$\lim_{t \to \infty}\frac{\sum_{k=0}^\infty a_kt^k}{\sum_{k=0}^\infty b_kt^k}=0\,.$$
\end{lem}
\begin{proof}
Let $\varepsilon > 0$ be arbitrary and let $N>0$ be such that 
$\frac{a_k}{b_k} < \varepsilon$ for all $k>N$. Also, let
$g(t) = \sum_{k=0}^\infty a_kt^k$ and $h(t) = \sum_{k=0}^\infty b_kt^k$.
We have
\begin{equation}\label{sum}
\frac{g(t)}{h(t)} = \frac{T_N^g(t) + g_1(t)}{T_N^h(t) + h_1(t)}
 = \frac{T_N^g(t)}{T_N^h(t) + h_1(t)}
                           + \frac{g_1(t)}{T_N^h(t) + h_1(t)},
\end{equation}
with $T_N^g(t) = \sum_{k=0}^N a_kt^k $, $T_N^h(t) = \sum_{k=0}^N b_kt^k $ 
and $g_1(t)$, $h_1(t)$ being the tails of the corresponding series.
The first term on the right-hand side of (\ref{sum}) manifestly tends to zero
as $t\to \infty$. The second term is clearly bounded above by
$g_1(t)/h_1(t)$.
The result then follows from 
$$g_1(t) = \sum_{k=N+1}^\infty a_kt^k = \sum_{k=N+1}^\infty \left (\frac{a_k}{b_k}
\right )b_k t^k
< \varepsilon\, h_1(t)$$
and the fact that $\varepsilon$ is arbitrary.
\end{proof}

\vskip 0.05in

\begin{lem}\label{Lemma3}
Let $\lambda_1, \lambda_2, \ldots ,\lambda_n\in \complex$
be given, with $\lambda_1 > |\lambda_2|\ge \cdots \ge |\lambda_n|$,
and let $f\in {\cal P}_\infty \cup {\cal P}^\infty$ be defined 
at these points. 
Then
\begin{equation}\label{Asy}
\lim_{t \to t^* -} \frac{t^jf^{(j)}(t \lambda_i)}{f(t \lambda_1)} = 0,
\quad {\rm for}\quad j=0,1,\ldots , \quad i = 2,\ldots n\,,
\end{equation}
where $t^* = R_f/\lambda_1$ and $R_f$ is the radius of convergence
of the series defining $f$ around $0$ (finite or infinite according to
whether $f\in {\cal P}^\infty$ or $f\in {\cal P}_\infty$, respectively).
\end{lem}
\begin{proof}
Consider first the case $t^* < \infty$. In this case the assumption
that $f\in {\cal P}^{\infty}$ 
guarantees (cf.~(\ref{cond})) that the denominator of
(\ref{Asy}) tends to infinity, whereas the numerator
remains finite for all $i\ne 1$ and all $j$. 
Indeed, each derivative $f^{(j)}(z)$ of $f(z)$ can
be expressed by a power
series having the same radius of convergence as the power series
expressing $f(z)$. Since each 
$t^*\lambda_i$ (with $i\ne 1$) falls inside the circle of convergence, we
have $|f^{(j)}(t^*\lambda_i)| < \infty$ for each $j\ge 0$, hence (\ref{Asy}).

Next, we consider the case where $t^* = \infty$. 
Let $i\ne 1$ and assume $\lambda_i \ne 0$ (the result is trivial for
$\lambda_i =0$). 
Since $f$ is entire, so are all its derivatives and moreover 
\begin{equation}\label{ineq1}
|f^{(j)}(t \lambda_i)| = \left |\sum_{k=0}^\infty (k+j)^{\underline j}\, c_{k+j} t^k 
\lambda_i^k \right | 
\le 
\sum_{k=0}^\infty (k+j)^{\underline j}\, c_{k+j} t^k |\lambda_i|^k < \infty\,,
\end{equation}
where we have used the (standard) notation $(k+j)^{\underline j}=
(k+j)(k+j-1)\cdots (k+1)$ (with the convention $k^{\underline 0}=1$).
Now let $a_k = (k+j)^{\underline j}\, c_{k+j} \lambda_i^k$ be the
coefficient of $t^k$
in the power series expansion of
$t^j f^{(j)}(t\lambda_i)$ and let 
$b_k = c_k \lambda_i^k$ be the
coefficient of $t^k$
in the power series expansion of
$f(t\lambda_1)$, then
\begin{equation} \label{ratio}
\frac{a_k}{b_k} = (k+j)^{\underline j} 
\left (\frac{|\lambda_i|}{\lambda_1}\right )^k, \quad {\rm for\,\, all}\quad k\ge j.
\end{equation}
Since exponential decay trumps polynomial growth, we
conclude that the expression in (\ref{ratio}) tends to zero as $k\to \infty$.
Using Lemma \ref{Lemma2} we obtain the desired conclusion.
\end{proof}

\vskip 0.05in

As we will see in the next section, the limit (\ref{Asy})
with $j=0$ will be instrumental in our analysis of
undirected networks, while the general case is needed for the
analysis of directed networks.

\section{Limiting behavior of parameterized centrality measures}
\label{sec:undirected_cent}
One difficulty in measuring the ``importance'' of a node in a 
network is that it is not always clear which 
of the many centrality measures should be used. Additionally, 
it is not clear a priori when two centrality measures will 
give similar node rankings on a given network.  When using 
parameter-dependent indices, such as 
Katz, exponential, or resolvent-based subgraph centrality, the necessity 
of choosing the value of the parameter 
adds another layer of difficulty. For instance, it is well known
that using different choices of $\alpha$ 
and $\beta$ in Katz and subgraph centrality
will generally produce different centrality scores 
and can lead to different node rankings. However, experimentally, 
it has been observed that different centrality measures often provide 
rankings that are highly correlated 
\cite{BKlym13,Dingetal,FR,Lee,MiPapa}. Moreover, in most cases, the 
{\em rankings} are quite stable, in the sense that
they do not appear to change much for different choices of 
$\alpha$ and $\beta$, even if the actual {\em scores}
may vary by orders of magnitude \cite{Klymko13}. With Katz and subgraph
centrality this happens in particular when
the parameters
$\alpha$ and $\beta$ approach their limits:
$$\alpha \to 0+,\quad \alpha \to \frac{1}{\lambda_1}-, \quad
\beta \to 0+, \quad \beta \to \infty\,.$$
Noting that the first derivatives of the node centrality measures
grow unboundedly as $\alpha \to \frac{1}{\lambda_1}-$ and as
$\beta \to \infty$, the centrality scores are extremely sensitive
to (vary extremely rapidly with) small changes in $\alpha$
when $\alpha$ is close to $\frac{1}{\lambda_1}-$, and in
$\beta$ when $\beta$ is even moderately large. Yet, the
rankings produced stabilize quickly and do not change much 
(if at all) when $\alpha$ and $\beta$
approach these limits. The
same is observed as $\alpha, \beta \to 0+$ .

The remainder of this section is devoted to proving
that the same behavior can be expected, more generally, 
when using  parameterized centrality measures based on analytic functions
$f\in {\cal P}$. The observed behavior for Katz and subgraph centrality
measures is thus explained and generalized.

It is worth noting that while all the parameterized centrality measures
considered here depend continuously on $t\in [0,t^*)$, the
rankings do not: hence, the limiting behavior of the ranking as the
parameter tends to zero cannot be obtained by simply setting the
parameter to zero.

\subsection{Undirected networks}\label{sec:undir}
We begin with the undirected case. The following theorem
is our  main result.
It completely describes the limiting behavior, for
``small" and ``large" values of the parameter, of parameterized
functional centrality measures based on either the diagonal entries
or the row sums.
Recall that a nonnegative matrix $A$ is {\em primitive} if
$\lambda_1>|\lambda_i|$
for $i=2,\ldots ,n$; see, e.g., \cite[p.~674]{Meyer00}. 

\vskip 0.05in

\begin{theorem}\label{Thm1}
Let $G=(V,E)$ be a connected, undirected, unweighted network with 
adjacency matrix $A$, assumed to be primitive, and let $f\in \cal P$ be defined on
the spectrum of $A$. Let $SC_i(t) = [f(tA)]_{ii}$ 
be the $f$-subgraph centrality of node $i$ and let ${\bf SC}(t)$ be the 
corresponding vector of $f$-subgraph centralities. 
Also, let $TC_i(t) = [f(tA){\bf 1}]_i$ be the total
$f$-communicability of node $i$
and let ${\bf TC}(t)$ be the corresponding vector.
Then,
\begin{enumerate}[(i)]
\item as $t \to 0+$, the rankings produced by both ${\bf SC}(t)$ 
and ${\bf TC}(t)$
converge to those produced by ${\bf d} = (d_i)$, the vector of degree 
centralities;  			
\item if in addition $f\in {\cal P}_\infty \cup {\cal P}^\infty$,
 then for $t \to t^* -$ the rankings produced by both 
${\bf SC}(t)$ and ${\bf TC}(t)$ converge to those produced by 
eigenvector centrality, i.e., by the entries of ${\bf q}_1$, 
the dominant eigenvector of $A$; 
\item the conclusion in (ii) still holds if the vector of all ones $\bf 1$
is replaced by any preference vector ${\bf v} > {\bf 0}$ in the 
definition of ${\bf TC}(t)$.
\\
\end{enumerate}
\end{theorem}

\begin{proof}
To prove (i), consider first the Maclaurin expansion of $SC_i(t)$: 
$$SC_i(t) = c_0 + c_1 t [A]_{ii} + c_2 t^2 [A^2]_{ii} + c_3 t^3 [A^3]_{ii}
+ \cdots = c_0 + 0 + c_2 t^2 d_i + c_3 t^3 [A^3]_{ii} + \cdots $$ 
Let $\boldsymbol\phi(t):= \frac{1}{c_2t^2}[{\bf SC}(t) -c_0{\bf 1}]$. 
The rankings produced by $\boldsymbol\phi(t)$ will be the 
same as those produced by ${\bf SC}(t)$, as the scores for 
each node have all been shifted and scaled in the same way. Now, 
the $i$th entry of $\boldsymbol\phi(t)$ is given by
\begin{equation}\label{maclaurin}
\phi_i(t) = \frac{1}{c_2t^2}[SC_i(t) -c_0] =  d_i + 
\frac{c_3}{c_2}t [A^3]_{ii} + \frac{c_4}{c_2}t^2[A^4]_{ii} + \cdots ,
\end{equation}
which tends to $d_i$ as $t\to 0+$.  Thus, as $t \to 0+$, 
the rankings produced by the $f$-subgraph centrality scores reduce to 
those produced by the degrees.

Similarly, we have
\begin{equation}\label{maclaurin2}
TC_i(t) = [f(tA) {\bf 1}]_i = [c_0{\bf 1} + c_1 t A{\bf 1} + c_2 t^2 A^2 {\bf 1} + \cdots ]_i
= c_0 + c_1 t d_i + c_2 t^2 [A{\bf d}]_i + \cdots
\end{equation}
Subtracting $c_0$ from $[f(tA) {\bf 1}]_i$ and dividing the result by $c_1 t$
leaves the quantity $d_i + O(t)$, hence for $t\to 0+$ we obtain again degree
centrality.

To prove (ii), consider first the expansion of $SC_i(t)$ in terms of 
the eigenvalues and eigenvectors of $A$: 
$$SC_i(t) = \sum_{k=1}^n f(t\lambda_k) 
q_k(i)^2 = f(t\lambda_1) q_1(i)^2 
+\sum_{k=2}^n f(t\lambda_k)  q_k(i)^2,$$  
where $q_k(i)$ is the $i$th entry of the (normalized) eigenvector 
${\bf q}_k$ of $A$ associated with $\lambda_k$.
Let $\boldsymbol\psi(t) := \frac{1}{f(t\lambda_1)}
{\bf SC}(t)$.  As in the proof of (i), the rankings produced by 
$\boldsymbol\psi(t)$ are the same as those produced by 
${\bf SC}(t)$, since the scores for each node have all been rescaled 
by the same amount.  Next, the $i$th entry of $\boldsymbol\psi(t)$ is 
\begin{equation}\label{diff1}
\psi_i(t) = q_1(i)^2 + \sum_{k=2}^n 
\frac{f(t\lambda_k)}{f(t \lambda_1)} q_k(i)^2.  
\end{equation}
Since $A$ is primitive, we have $\lambda_1 > \lambda_k$ for $2\le k\le n$.
Hence, applying Lemma \ref{Lemma3} with $j=0$ we conclude that
$\psi_i(t) \to q_1(i)^2$ as $t\to t^* -$.  
By the Perron-Frobenius Theorem we can choose ${\bf q}_1 >\bf 0$, hence the rankings 
produced by $q_1(i)^2$ are the same as those produced by 
$q_1(i)$.  Thus, as $t \to t^* -$, the rankings produced by 
the $f$-subgraph centrality scores reduce to those obtained with
eigenvector centrality.

Similarly, we have
\begin{equation}\label{diff2}
TC_i(t) = \sum_{k=1}^n f(t\lambda_k)({\bf q}_k^T{\bf 1})q_k(i)
= f(t\lambda_1)({\bf q}_1^T{\bf 1})q_1(i) + 
\sum_{k=2}^n f(t\lambda_k)({\bf q}_k^T{\bf 1}) 
q_k(i).
\end{equation}
Note that ${\bf q}_1^T{\bf 1} > 0$ since ${\bf q}_1 >\bf 0$. Dividing both sides by 
$f(t\lambda_1) {\bf q}_1^T {\bf 1}$ and taking the limit as $t\to t^* -$ we obtain the
desired result.

Finally, (iii) follows by just replacing $\bf 1$ with $\bf v$ in the
foregoing argument.
\end{proof} 

\vskip 0.03in

By specializing the choice of $f$ to the matrix exponential and 
resolvent, we immediately obtain the following corollary of Theorem
\ref{Thm1}.

\vskip 0.03in

\begin{cor} 
Let $G=(V,E)$ be a connected, undirected, unweighted network with 
adjacency matrix $A$, assumed to be primitive. 
Let $EC_i(\beta) = [{\rm e}^{\beta A}]_{ii}$ and
$RC_i(\alpha) = [(I - \alpha A)^{-1}]_{ii}$ 
be the exponential and resolvent subgraph centralities of node $i$.  
Also, let 
$TC_i(\beta) = [{\rm e}^{\beta A}{\bf 1}]_{i}$ and
$K_i(\alpha) = [(I - \alpha A)^{-1}{\bf 1}]_{i}$ be the total communicability 
and Katz centrality
of node $i$, respectively.
Then, the limits in table \ref{Corollary_1} hold.
Moreover, the limits for $TC_i(\beta)$ and $K_i(\alpha)$ remain the same
if the vector $\bf 1$ is replaced by an arbitrary preference vector ${\bf v} > {\bf 0}$.\\
\label{Cor1}
\end{cor}

\begin{table}[t]
 \caption{Limiting behavior of different ranking schemes, undirected case.}
 \label{Corollary_1}
\begin{center}
{
\begin{tabular}{|c| cc |} \hline
 & \multicolumn{2}{c|}{Limiting ranking scheme}\\
\cline{2-3}
Method  & degree  & eigenvector \\
\hline
$RC(\alpha)$, $K(\alpha)$ & $\alpha \to 0+$ & $\alpha \to \frac{1}{\lambda_1}-$ \\
$EC(\beta)$, $TC(\beta)$  & $\beta \to 0+$  & $\beta \to \infty$   \\
\hline
\end{tabular}
}
\end{center}
\end{table}

\begin{remark} \label{WLOG}
{\rm
The restriction to primitive matrices is required
in order to have $\lambda_1 > \lambda_k$ for $k\ne 1$, so 
that Lemma \ref{Lemma3} can be used in the proof of Theorem \ref{Thm1}.
At first sight, this assumption may seem
somewhat restrictive; for instance, bipartite
graphs would be excluded, since they have $\lambda_n = -\lambda_1$. 
In practice, however, there is no loss of generality. Indeed, if 
$A$ is imprimitive we can replace $A$ with the (always primitive)
matrix $A_\varepsilon = (1-\varepsilon)A + \varepsilon I$ with 
$0< \varepsilon < 1$, compute the quantities of interest
using $f(tA_\varepsilon)$, and then let $\varepsilon \to 0$. Note that 
$\rho(A_\varepsilon) = \rho(A)$, hence the radius of convergence is unchanged. 
Also note that for some centrality measures, such as those based on the matrix
exponential, it is not
even necessary to take the limit for $\varepsilon \to 0$. Indeed, we have
${\rm e}^{\beta A_\varepsilon} = 
{\rm e}^{\beta \varepsilon} {\rm e}^{\beta (1-\varepsilon)A}.$
The prefactor ${\rm e}^{\beta \varepsilon}$ is just a scaling that does not affect
the rankings, and ${\rm e}^{\beta (1-\varepsilon)A}$  and ${\rm e}^{\beta A}$ and
have identical limiting behavior for $\beta\to 0$ or $\beta\to \infty$. 
}
\end{remark}

\subsection{Directed networks}\label{sec:dir}
Here we extend our analysis to directed networks. 
The discussion is similar to the one for the undirected case, except that
now we need to distinguish between receive and broadcast centralities.
Also, the Jordan canonical
form must replace the spectral decomposition in the proofs.

\vskip 0.05in

\begin{theorem} \label{Thm2}
Let $G=(V,E)$ be a strongly connected, directed, unweighted network with 
adjacency matrix $A$, 
and let $f\in \cal P$ be defined on the
spectrum of $A$. 
 Let $TC^{b}_i(t) =
[f(tA) {\bf 1}]_{i}$ be the broadcast total
$f$-communicability of node $i$ and ${\bf TC}^b(t)$ be the
corresponding vector of broadcast total $f$-communicabilities.
Furthermore,
let $TC^r_i(t) =  [f(t A^T){\bf 1}]_{i} = [f(tA)^T {\bf 1}]_{i}$ be the
receive total $f$-communicability of node $i$ and ${\bf TC}^r(t)$
be the corresponding vector of receive total $f$-communicabilities.
Then,
\begin{enumerate}[(i)]
\item as $t \to 0+$, the rankings produced by
${\bf TC}^b(t)$ converge to those produced by
the out-degrees of the nodes in the network;
\item as $t \to 0+$, the rankings produced by
${\bf TC}^r(t)$ converge to those produced by
the in-degrees of the nodes in the network;
\item if $f\in {\cal P}_\infty \cup
{\cal P}^\infty$, then as $t \to t^* -$, the rankings produced by
${\bf TC}^b(t)$ converge to those produced by ${\bf x}_1$,
where ${\bf x}_1$ is the dominant right eigenvector of $A$;
\item if $f\in {\cal P}_\infty \cup
{\cal P}^\infty$, then as $t \to t^* -$, the rankings produced by
${\bf TC}^r(t)$ converge to those produced by
${\bf y}_1$, where ${\bf y}_1$ is the dominant left eigenvector of $A$;
\item results (iii) and (iv) still hold if ${\bf 1}$ is replaced by an arbitrary 
preference vector ${\bf v} > {\bf 0}$ in the definitions of
${\bf TC}^b(t)$ and ${\bf TC}^r(t)$.
\\
\end{enumerate}
\label{thm:expTC_dir}
\end{theorem}

\begin{proof}
The proofs of (i) and (ii) are analogous to that 
for $TC_i(t)$
in part (i) of Theorem \ref{Thm1}, keeping in mind that the entries
of $A{\bf 1}$ are the out-degrees and those of $A^T{\bf 1}$ are
the in-degrees of the nodes of $G$.

To prove (iii), observe that if $f$ is defined on the spectrum of $A$,
then
\begin{equation}\label{general}
f(A) = \sum_{k=1}^s \sum_{j=0}^{n_k - 1} \frac{f^{(j)} (\lambda_k)}{j!} (A-\lambda_k I)^j G_k,
\end{equation}
where $s$ is the number of distinct eigenvalues of $A$, $n_k$ is the {\rm index}
of the eigenvalue $\lambda_k$ (that is, the order of the largest Jordan block 
associated with $\lambda_k$
in the Jordan canonical form of $A$), and $G_k$ is the oblique projector
with range $\mathcal{R}(G_k)=\mathcal{N}((A-\lambda_kI)^{n_k})$ and null space
$\mathcal{N}(G_k)=\mathcal{R}((A-\lambda_kI)^{n_k})$; see, e.g., \cite[Sec.~1.2.2]{High}
or \cite[Sec.~7.9]{Meyer00}. 
Using (\ref{general})
and the fact that $\lambda_1$ is
simple by the Perron-Frobenius theorem, we find
$$TC_i^b(t) = f(t\lambda_1) ({\bf y}_1^T{\bf 1}) x_1(i)
+
\sum_{k=2}^s \sum_{j=0}^{n_k-1} \frac{t^j f^{(j)}(t\lambda_k)}{j!}
[(A-\lambda_kI)^j G_k{\bf 1}]_{i}.$$
Noting that ${\bf y}_1^T{\bf 1} > 0$, let $\boldsymbol\psi^b(t) := \frac{1}
{f(t\lambda_1) ({\bf y}_1^T{\bf 1})}{\bf TC}^b(t)$.
The rankings produced by $\boldsymbol\psi^b(t)$ will be the same
as those produced by ${\bf TC}^b(t)$. Now,
the $i$th entry of $\boldsymbol\psi^b(t)$ is
\begin{equation}\label{psi2}
\psi^b_i(t) = x_1(i) + \sum_{k=2}^s \sum_{j=0}^{n_k-1}
\frac{t^j f^{(j)}(t \lambda_k)}{j! f(t\lambda_1)({\bf y}_1^T{\bf 1})}
[(A-\lambda_kI)^j G_k{\bf 1}]_{i}.
\end{equation}
Without loss of generality, we can assume that $\lambda_1 > 
|\lambda_k|$ for $k\ne 1$ (see Remark \ref{WLOG}). 
By Lemma \ref{Lemma3} the second term on the right-hand side of 
(\ref{psi2}) vanishes as $t\to t^*-$, 
and therefore $\psi^b_i(t)\to x_1(i)$; that is,
the rankings given by $\boldsymbol\psi^b(t)$
reduce to those given by the right dominant eigenvector ${\bf x}_1$ of $A$
in the limit $t\to t^*-$.
 
The proof of (iv) is completely analogous to that of (iii).

Finally, the proof of (v) is obtained by replacing $\bf 1$
with $\bf v$ and observing that the argument used to prove (iii)
(and thus (iv))
remains valid.
\end{proof}

\vskip 0.05in

By specializing the choice of $f$ to the matrix exponential and 
resolvent, we immediately obtain the following corollary of Theorem
\ref{Thm2}.\\

\begin{cor}
Let $G=(V,E)$ be a strongly connected, directed, unweighted network with
adjacency matrix $A$. 
Let
$EC_i^b(\beta) = [{\rm e}^{\beta A}{\bf v}]_{i}$ and
$K_i^b(\alpha) = [(I - \alpha A)^{-1}{\bf v}]_{i}$ be the total communicability
and Katz broadcast centrality
of node $i$, respectively. Similarly, let
$EC_i^r(\beta) = [{\rm e}^{\beta A^T}{\bf v}]_{i}$ and
$K_i^r(\alpha) = [(I - \alpha A^T)^{-1}{\bf v}]_{i}$ be the total communicability
and Katz receive centrality of node $i$.
Then, the limits in Table \ref{Corollary_2} hold.\\
Moreover, all these limits remain the same if the vector $\bf 1$ is replaced by
an arbitrary preference vector ${\bf v} > \bf 0$.
\label{Cor2}
\end{cor}

\begin{table}[t]
 \caption{Limiting behavior of different ranking schemes, directed case.}
 \label{Corollary_2}
\begin{center}
{
\begin{tabular}{|c| cccc |} \hline
 & \multicolumn{4}{c|}{Limiting ranking scheme}\\
\cline{2-5}
Method  & out-degree  & in-degree & right eigenvector & left eigenvector \\
\hline
$K^b(\alpha)$ & $\alpha \to 0+$ &  & $\alpha \to \frac{1}{\lambda_1}-$ &  \\
$K^r(\alpha)$ &              & $\alpha \to 0+$ &  & $\alpha \to \frac{1}{\lambda_1}-$\\
$EC^b(\beta)$ & $\beta \to 0+$  &  & $\beta \to \infty$ &  \\
$EC^r(\beta)$ &  & $\beta \to 0+$  &  & $\beta \to \infty$ \\
\hline
\end{tabular}
}
\end{center}
\end{table}

\vskip 0.05in

This concludes our analysis in the case of simple, strongly connected (di)graphs.

\subsection{Extensions to more general graphs} \label{sec:ext}
So far we have restricted our discussion to unweighted, loopless graphs.
This was done in part for ease of exposition. Indeed, it is easy to see
that all of the results in Theorem \ref{Thm1} and Corollary \ref{Cor1}
remain valid in the case of {\em weighted} undirected
networks if all the weights $a_{ij}$ (with $(i,j)\in E$) are positive and if
we interpret the degree of node $i$ to be the
{\em weighted degree}, i.e., the $i$th row sum $A{\bf 1}$.
The only case that cannot be generalized is that relative to
${\bf SC}(t)$ as $t\to 0+$  in Theorem \ref{Thm1} and, as
a consequence, those relative to $EC_i(\beta)$ and $K_i(\alpha)$ 
as $\beta, \alpha \to 0+$ in Corollary \ref{Cor1}.
The reason for this is that in general it is no longer true that
$[A^2]_{ii} = d_i$, i.e., the diagonal entries of $A^2$ are not generally equal
to the weighted degrees.

Furthermore, all of the results in Theorem \ref{Thm2} and Corollary \ref{Cor2}
remain valid in the case of strongly connected,
{\em weighted} directed networks if we interpret the 
out-degree and in-degree of node $i$ as weighted out- and in-degree,
given by the $i$th row and column sum of $A$, respectively. 

Finally, all the results relative to the limit
$t\to t^*-$ in Theorems \ref{Thm1} and \ref{Thm2} remain valid in the
presence of loops (i.e., if $a_{ii}\ne 0$ for some $i$).
Hence, in particular, all the results in Corollaries
\ref{Cor1} and \ref{Cor2} concerning  the behavior of the various
exponential and resolvent-based centrality measures for $\beta\to \infty$
and $\alpha \to \frac{1}{\lambda_1}-$ remain valid in this case.

\section{The case of PageRank}\label{sec:PR}
In this section we discuss the limiting 
behavior of the PageRank algorithm \cite{Page}, which has a well known
interpretation in terms of random walks on a digraph (see, e.g.,
\cite{Pagerank}). Because of the special structure possessed by
the matrices arising in this method, a somewhat different treatment than
the one developed in the previous section is required. 

Let $G=(V,E)$ be an arbitrary digraph with $|V|=n$ nodes, and
let $A$ be the corresponding adjacency matrix. From $A$ we construct
an irreducible, column-stochastic matrix $P$ as follows.
Let $D$ be the diagonal matrix with entries 
$$d_{ij}=\left\{\begin{array}{ll}
d_i^{out},& \textnormal{ if }\,\, i=j \quad \textnormal{and }\quad d_i^{out}>0,\\
1, & \textnormal{ if }\,\, i=j \quad \textnormal{and } \quad d_i^{out}=0 ,\\
0, & \textnormal{ else. }
\end{array}\right .
$$
Now, let 
\begin{equation}\label{matrix_H}
H=A^TD^{-1}. 
\end{equation}
This matrix may have zero columns, corresponding
to those indices $i$ for which $d_i^{out}=0$; the corresponding nodes
of $G$ are known as {\em dangling nodes}. Let $I$ denote the set
of such indices, and define the vector ${\bf a} = (a_i)$ by
$$ a_i =\left\{\begin{array}{ll}
1,& \textnormal{ if }\,\, i\in I,\\
0, & \textnormal{ else. }
\end{array}\right .
$$
Next, we define the matrix $S$ by 
\begin{equation}\label{matrix_S}
S = H + \frac{1}{n}{\bf 1}\,{\bf a}^T.
\end{equation}
Thus, $S$ is obtained from $H$ by
replacing each zero column of $H$ (if present)
by the column vector $\frac{1}{n}{\bf 1}$.
Note that $S$ is column-stochastic, but could still be
(and very often is) reducible. To obtain an irreducible matrix, we take
$\alpha \in (0,1)$ and construct the ``Google matrix"
\begin{equation}\label{google}
P = \alpha S + (1-\alpha){\bf v}{\bf 1}^T\,,
\end{equation}
where ${\bf v}$ is an arbitrary probability distribution vector (i.e., a
column vector with nonnegative entries summing up to 1). 
The simplest choice for $\bf v$ is the uniform distribution, 
${\bf v} = \frac{1}{n}\bf 1$,
but other choices are possible.
Thus, $P$ is a convex combination of the modified
scaled adjacency matrix $S$ and a rank-one matrix, and is
column-stochastic. 
If every entry in $\bf v$ is strictly positive (${\bf v} > \bf 0$),
$P$ is also positive and therefore acyclic and irreducible.
The Markov chain associated with $P$ is {\em ergodic}: it has
a unique steady-state probability distribution 
vector ${\bf p} = {\bf p}(\alpha) > \bf 0$, given by the dominant 
eigenvector of $P$, normalized so that ${\bf p}^T{\bf 1} = 1$: thus,
$\bf p$ satisfies ${\bf p} = P\,{\bf p}$, or $(I - P)\,{\bf p} = \bf 0$ . 
The vector  $\bf p$ is known as the {\em PageRank vector}, and it can
be used to rank the nodes in the original digraph $G$.
The success of this method in ranking web pages is universally
recognized. It has also been used successfully in many other settings.

The role of the parameter $\alpha$ is to
balance the structure of the underlying digraph with the probability of
choosing a node at random 
(according to the probability distribution $\bf v$)
in the course of a random walk on the graph. Another important consideration
is the rate of convergence to steady-state of the Markov chain: the smaller
is the value of $\alpha$, the faster the convergence. 
In practice, the choice $\alpha = 0.85$
is often recommended.

It was recognized early on that the PageRank vector can also be obtained
by solving a non-homogeneous linear system of equations. In fact, there is more
than one such linear system; see, e.g., \cite[Chapter 7]{Pagerank} and
the references therein.
One possible reformulation of the problem is given by
the linear system
\begin{equation}\label{lin_sys}
(I - \alpha H){\bf x} = {\bf v}, \quad {\bf p} = {\bf x}/({\bf x}^T{\bf 1}).
\end{equation}
For each $\alpha \in (0,1)$, the coefficient matrix in (\ref{lin_sys}) is
a nonsingular $M$-matrix, hence it is invertible with a nonnegative inverse
$(I - \alpha H)^{-1}$.
Note the similarity of this linear system with the one corresponding to
Katz centrality. Using this equivalence, we can easily describe the
limiting behavior of PageRank for $\alpha \to 0+$.

\begin{theorem}\label{thm:PR0}
Let $H$ be the matrix defined in (\ref{matrix_H}), and let ${\bf p}(\alpha)$
be the PageRank vector corresponding 
to a given $\alpha\in (0,1)$. Assume ${\bf v} = \frac{1}{n}{\bf 1}$ in the
definition (\ref{google}) of the Google matrix $P$.
Then,
for $\alpha \to 0+$, the rankings given by ${\bf p}(\alpha)$
converge to those given by the vector 
$H{\bf 1}$, the row sums of $H$ or, equivalently, by the vector $S{\bf 1}$,
the row sums of $S$. 
%
\end{theorem}
\begin{proof}
Note that for each
$\alpha \in (0,1)$ 
the inverse matrix $(I - \alpha H)^{-1}$ can be
expanded in a Neumann series, hence the unique
solution of (\ref{lin_sys}) can be expressed as
\begin{equation}\label{neumann}
{\bf x} = {\bf v} + \alpha H {\bf v} + \alpha^2 H^2 {\bf v} + \cdots 
\end{equation}
When ${\bf v}=\frac{1}{n} {\bf 1}$, the rankings given by the entries
of $\bf x$ coincide with those given by the entries of $({\bf x} - {\bf v})/\alpha$.
But (\ref{neumann}) implies
$$\lim_{\alpha \to 0+} \frac{{\bf x} - {\bf v} } {\alpha } =  
H {\bf v} = \frac{1}{n} H {\bf 1},$$
showing that the rankings from ${\bf p}(\alpha)$ coincide with those
from the row sums of $H$ in the limit
$\alpha \to 0+$. Finally, $S{\bf 1} = H {\bf 1} + ({\bf a}^T{\bf 1}/n){\bf 1}$,
hence the row sums of $S$ result in the same limit rankings.
\end{proof}

We emphasize that the above result only holds for the case of a uniform
personalization vector $\bf v$. 

\begin{remark}
{\rm 
Since $H$ is a scaled adjacency matrix, each entry of $H{\bf 1}$
is essentially a weighted in-degree; see also the discussion in section 
\ref{sec:ext}.
Hence, we conclude that the structure of the graph $G$ retains considerable
influence on the rankings obtained by PageRank
even for very small $\alpha$, as long as it is nonzero.\footnote{See
the Supplementary Materials to this paper for a numerical illustration
of this statement.}
}
\end{remark}

\begin{remark}
{\rm
The behavior of the PageRank vector for $\alpha \to 1$ (from the left) has 
received a great deal of attention in the literature; see, e.g.,
\cite{Boldi05,Boldi09,Pagerank,Vigna}.
Assume that $\lambda = 1$ is the only eigenvalue of $S$ on the unit circle.
Then it can be shown (see \cite{Meyer74,Vigna}) that for $\alpha \to 1-$, the rankings
obtained with PageRank with initial vector $\bf v$ converge to those given by the 
vector
\begin{equation}\label{GI}
 {\bf x}^* = (I - (I - S)(I - S)^{\sharp}){\bf v},
\end{equation} 
where $(I - S)^{\sharp}$ denotes the group generalized inverse of
$I- S$ (see \cite{CM}).
In the case of a uniform personalization vector ${\bf v} = \frac{1}{n} \bf 1$,
(\ref{GI}) is equivalent to using
the row sums of the matrix $(I - S)(I - S)^{\sharp}$.  
As discussed in detail in \cite{Boldi09}, however, 
using this vector
may lead to rankings that are not very meaningful, since when $G$ is not
strongly connected (which is usually the case in practice), it tends to
give zero scores
to nodes that are arguably the most important.
For this reason, values of $\alpha$ too close to 1 are not recommended.
}
\end{remark}

We conclude this section with a few remarks about another technique,
known as {\em DiffusionRank} \cite{YKL} or {\em Heat Kernel PageRank}
\cite{Chung}. This method is based on the matrix
exponential ${\rm e}^{tP}$, where $P$ is column-stochastic, acyclic
and irreducible. For 
example, $P$ could be the ``Google" matrix constructed
from a digraph $G$ in the manner described above. It is immediate to see that
for all $t>0$
the column sums of ${\rm e}^{tP}$ are all equal to ${\rm e}^{t}$, hence
the scaled matrix ${\rm e}^{-t}{\rm e}^{tP}$ is column-stochastic. Moreover,
its dominant eigenvector is the same as the dominant eigenvector of $P$,
namely, the PageRank vector. It follows from the results found in
section \ref{sec:dir}, and can easily
be shown directly (see, e.g., \cite{YKL}), that the node rankings
obtained using the row sums
of ${\rm e}^{tP}$ tend, for $t\to \infty$, to those given by PageRank. 
Hence, the PageRank vector can be regarded as the equilibrium distribution
of a continuous-time diffusion process on the underlying digraph.

\section{Discussion}
\label{sec:interpretation}
The centrality measures considered in this paper are 
all based on walks on the network.  The degree centrality of a node 
$i$ counts the number of walks of length one starting at $i$ 
(the degree of $i$).  In contrast, the eigenvector centrality of 
node $i$ gives the limit as $k\to \infty$  of the percentage 
of walks of length $k$ which start at node $i$ among all walks
of length $k$ (see \cite[Thm.~2.2.4]{CvRoSi97} and \cite[p.~127]{Ebook11}). 
 Thus, the degree centrality of node $i$ measures 
the {\em local} influence of $i$ and the eigenvector centrality measures 
the {\em global} influence of $i$.  

When a centrality measure associated with an analytic function
$f\in \cal P$ is used, walks of all lenghts are included in
the calculation of centrality scores, and a weight $c_k$ is assigned to the  
walks of length $k$, where $c_k \to 0$ as $k\to \infty$. Hence,
both local and global influence are now taken into account, but with longer walks 
being penalized more heavily than shorter ones. 
The parameter $t$ permits further tuning of the weights;
as $t$ is decreased,
the weights corresponding to larger $k$ decay faster and shorter walks
become more important.  In the limit as
$t\to 0+$, walks of length one (i.e., edges) dominate the centrality scores
and the rankings converge to the degree centrality rankings.  As
$t$ is increased, given a fixed walk length $k$, the
corresponding weight
increases more rapidly than those of shorter walks.  In the limit
as $t \to t^*-$, walks of ``infinite" length dominate and the
centrality rankings converge to those of eigenvector centrality.

Hence, when using
parameterized centrality measures, the parameter $t$ can be regarded 
as a ``knob" that can be used for interpolating, or tuning, 
between rankings based on local influence (short walks) and 
those based on global influence (long walks).  In applications 
where local influence is most important, degree centrality will 
often be difficult to distinguish from any of the parameterized centrality 
measures with $t$ small.  Similarly, when global 
influence is the only important factor, parameterized centrality 
measures with $t\approx t^*$ will often be virtually
indistinguishable from eigenvector centrality. 

Parameterized centrality measures are likely to be
most useful when both local and global 
influence need to be considered in the ranking of nodes in a 
network. In order to achieve this, ``moderate" values of 
$t$ (not too small and not too close to $t^*$) should be used. 

To make this notion more quantitative, however, we need some way to 
estimate how fast the limiting rankings given by degree and
eigenvector centrality are approached for $t\to 0+$
and $t\to t^*-$, respectively. We start by considering the undirected case 
(weights and loops are allowed).
The approach to the eigenvector centrality limit as $t\to t^*-$
depends on the {\em spectral gap} $\lambda_1 - \lambda_2$ of
the adjacency matrix of the network.
This is clearly seen from the fact that the difference between
the various parameterized centrality measures (suitably scaled)
depends on the ratios $\frac{f(t\lambda_k)}{f(t\lambda_1)}$, for
$2\le k\le n$; see
(\ref{diff1}) and (\ref{diff2}). Since a function $f\in \cal P$ is 
strictly increasing with $t$ (when $t>0$), a relatively large spectral gap
implies that each term containing $\frac{f(t\lambda_k)}{f(t\lambda_1)}$
(with $k\ne 1$) will tend rapidly to zero as $t\to t^*-$,
since $f(t\lambda_1)$ will grow much faster than $f(t\lambda_k)$.
For example, in the case of exponential subgraph centrality the
$k=2$ term in the sum contains the factor ${\rm e}^{\beta \lambda_2}
/{\rm e}^{\beta \lambda_1} = {\rm e}^{\beta (\lambda_2 - 
\lambda_1)}$, which decays to zero extremely fast for
$\beta \to \infty$ if $\lambda_1 - \lambda_2$
is ``large", with every other term with $k>2$ going to zero at least as fast.

More generally, when the spectral gap is large, the rankings
obtained using parameterized centrality will converge to those
given by eigenvector centrality more quickly as $t$ increases 
than in the case when the spectral gap is small.  
Thus, in networks with a large enough spectral gap, eigenvector 
centrality may as well be used instead of a measure based on the exponential 
or resolvent of the adjacency matrix.  However, it's not always easy
to tell {\em a priori} when $\lambda_1-\lambda_2$ is ``large enough''; 
some guidelines can be found in \cite{E06b}.
We also note that the tuning parameter $t$ can be
interpreted as a way to artificially widen or shrink the (absolute) gap, thus
giving more or less weight to the dominant eigenvector. 

The situation is rather more involved in the case of directed networks.
Equation (\ref{psi2}) shows that the difference between the
(scaled) parameterized centrality scores and the corresponding
eigenvector centrality scores contains terms of the form
$\frac{t^jf^{(j)}(t\lambda_k)}{j!f(t\lambda_1)}$ (with $0\le j\le n_k -1$,
where $n_k$ is the index of $\lambda_k$),
as well as additional quantities involving powers of $A-\lambda_k I$ and
the oblique projectors $G_k$. Although these terms
vanish as $t\to t^*-$, the spectral gap in this case can only
provide an {\em asymptotic} measure of how rapidly the eigenvector
centrality scores are approached, unless $A$ is nearly
normal.

Next, we turn to the limits as $t\to 0+$.  For brevity, we
limit our discussion to the undirected case. 
From equation (\ref{maclaurin}) we see that for small $t$, the difference
between the (scaled and shifted) $f$-subgraph centrality $\phi_i(t)$
of node $i$
and the degree $d_i$ is dominated by the term $\frac{c_3}{c_2}t[A^3]_{ii}$. Now,
it is well known that the number of triangles (cycles of length 3) that node $i$
participates in is equal to $\Delta_i = \frac{1}{2}[A^3]_{ii}$. It follows that
if a node $i$ participates in a large number of triangles, then
the corresponding centrality score $\phi_i(t)$ can be expected to approach  
the degree centrality score $d_i$ more slowly, for $t\to 0+$,
than a node $j$ that participates in no (or few) such triangles. 

To understand this intuitively, consider two nodes, $i$ and $j$, 
both of which have degree $k$.  Suppose node
$i$ participates in no triangles and node $j$ participates in
${k \choose 2}$ triangles.  That is, $N(i)$, the set of nodes adjacent to
node $i$, is an independent set of
$k$ nodes ({\em independent} means that no edges are present
between the nodes in $N(i)$), while $N(j)$ is a clique (complete
subgraph) of size $k$.  
In terms of local communities, node $i$ is isolated
(does not participate in a local community) while node $j$ sits 
at the center of a dense local community (a
clique of size $k+1$) and only participates in links to other nodes 
within this small, dense subgraph.  Due to this,
whenever $j$ communicates with any of its neighbors, this information 
can quickly be passed among all its
neighbors.  This allows the clique of size $k+1$ to act as a sort of 
``super-node'' where $j$'s local influence depends
greatly on the local influence of this super-node.  That is, even on a 
local level, it is difficult to separate the influence
of $j$ from that of its neighbors.  In contrast, node $i$ does not participate 
in a dense local community and, thus,
its local influence depends more on its immediate neighbors than on the 
neighbors of those neighbors.  Therefore,
local (i.e., small $t$) centrality measures on node $i$ will be more similar to degree 
centrality than those on node $j$.

From a more global perspective, we can expect the degree centrality
limit to be attained more rapidly, for $t\to 0+$, for
networks with low clustering coefficient\footnote{Recall
that the {\em clustering coefficient} of an undirected
network $G=(V,E)$ is defined
as the average of the {\em node clustering coefficients}
$CC(i):=\frac{2\Delta_i}{d_i(d_i -1)}$ over all nodes $i\in V$ of 
degree $d_i\ge 2$. See, e.g., \cite[p.~303]{Brandes}.} than for networks 
with high clustering coefficient (such as social networks).

For the total communicability centrality, on the other hand,
equation (\ref{maclaurin2}) suggests that the rate at which degree
centrality is approached is dictated, for small $t$, by the vector $A{\bf d} = A^2{\bf 1}$.
Hence, if node $i$ has a large number of next-to-nearest neighbors
(i.e., there are many nodes at distance 2 from $i$) then the degree
centrality will be approached more slowly, for $t\to 0+$, than for a
node that has no (or few) such next-to-nearest neighbors.

\section{Related work}
As mentioned in the Introduction,
correlations between the rankings obtained with different centrality measures,
such as degree and eigenvector centrality, have frequently been observed
in the literature.  A few authors have gone beyond this empirical
observation and have proved
rigorous mathematical statements explaining some of these 
correlations in special cases. Here we briefly review these previous
contributions and how they relate to our own.

Bonacich and Lloyd showed in \cite{BL01} that
eigenvector centrality is a limiting case of Katz centrality when 
$\alpha \to \frac{1}{\lambda}-$, but their proof assumes that $A$
is diagonalizable.

A centrality measure closely related to Katz centrality,  known as
(normalized) $\alpha$-centrality, was thoroughly studied in
\cite{GL11}. This measure actually depends on
{\em two} parameters $\alpha$ and $\beta$, and reduces to Katz centrality
when $\alpha = \beta$. The authors of \cite{GL11} show that 
$\alpha$-centrality reduces to degree centrality as $\alpha \to 0+$;
they also show, but only for {\em symmetric} adjacency matrices, that
$\alpha$-centrality reduces to eigenvector centrality for $\alpha \to \frac{1}{\lambda}-$
(a result less general than that about Katz centrality in \cite{BL01}). 

A proof that Katz centrality (with an arbitrary preference vector ${\bf v}$)
reduces to eigenvector centrality as $\alpha \to \frac{1}{\lambda}-$ for 
a general $A$ (that is, without requiring that $A$ be diagonalizable) can be
found in \cite{Vigna}. This proof avoids use of the Jordan canonical form but
makes use of the Drazin inverse, following 
\cite{Meyer74}. Unfortunately this technique is not easily generalized to 
centrality measures based on other matrix functions.

Related results can also be found in \cite{Shimbo}. In this paper,
the authors consider parameter-dependent matrices (``kernels") of the form
$$N_{\gamma}(B) = B\sum_{k=0}^\infty (\gamma B)^k = B(I - \gamma B)^{-1}
\quad {\rm and} \quad
E_{\gamma}(B) = {\rm e}^{\gamma B} = \sum_{k=0}^\infty \frac{(\gamma B)^k}{k!},$$
where $B$ is taken to be either $AA^T$ or $A^TA$, with $A$ the adjacency
matrix of a (directed) citation network. The authors show that 
$$\lim_{\gamma \to \gamma^* -} \left (\frac{1}{\rho (B)} - \gamma \right )
N_{\gamma}(B)={\bf v}{\bf v}^T,\quad \lim_{\gamma \to \infty}
{\rm e}^{-\gamma \rho(B)}E_{\gamma}(B) ={\bf v}{\bf v}^T,$$ 
where $\gamma^* = 1/\rho(B)$ and $\bf v$ is the dominant eigenvector of $B$. 
Noting that $\bf v$ is the hub vector when $B=AA^T$
and the authority vector when $B=A^TA$, the authors observe that the HITS
algorithm \cite{HITS} is a limiting case of the kernel-based algorithms.

Finally, we mention the work by Romance \cite{Romance}.
This paper introduces a general family of centrality measures which
includes as special cases degree centrality, eigenvector centrality, 
PageRank, $\alpha$-centrality (including Katz centrality), and many
others. Among other results, this general framework allows the author to 
explain the strong correlation between degree and eigenvector centrality
observed in certain networks, such as Erd\"os--Renyi graphs.
We emphasize that the unifying framework presented in \cite{Romance}
is quite different from ours.

In conclusion, our analysis allows us
to unify, extend, and complete some partial results
that can be found scattered in the literature concerning the relationship
among different centrality measures. In particular, our treatment covers
a broader class of centrality measures and networks than those considered by
earlier authors. In addition, we provide some rules of thumb
for the choice of parameters when using measures such as Katz and
subgraph centrality (see section \ref{sec:numerical_experiments}).

\section{ Summary of numerical experiments}
\label{sec:numerical_experiments}
In this section we briefly summarize the results of numerical 
experiments aimed at illustrating our theoretical results. A
complete description of the tests performed, inclusive of plots and tables,
can be found in the Supplementary Materials accompanying this paper.

We examined various parameterized centrality measures based on the matrix
exponential and resolvent, including subgraph and total communicability
measures.
Numerical tests were performed on a set of networks from different
application areas (social networks, preotein-protein interaction networks,
computer networks, collaboration networks, a road network, etc.). Both directed and
undirected networks were considered.
The tests were primarily aimed at monitoring the limiting behavior of
the various centrality measures for $\beta \to 0+$, $\beta \to \infty$
for exponential-type measures and for $\alpha \to 0+$, $\alpha \to \frac{1}{\lambda_1}-$
for resolvent-type measures.

Our experiments confirm that the rankings obtained with exponential-type centrality measures
approach quickly those obtained from degree centrality as $\beta$ gets smaller, with
the measure based on the diagonal entries $[{\rm e}^{\beta A}]_{ii}$ 
approaching degree centrality faster, in general, than 
the measure based on $[{\rm e}^{\beta A}{\bf 1}]_{i}$. The tests also
confirm that for networks with large spectral gap, the rankings obtained
by both of these measures
approach those from eigenvector centrality much more quickly, as $\beta$ increases,
than for the networks with small spectral gap. 
These remarks are especially true when only the top ranked nodes are
considered.

Similar considerations apply to resolvent-type centrality measures and
to directed networks. 

Based on our tests, we propose the following rules of thumb when using
exponential and resolvent-type centrality measures. For the matrix exponential 
the parameter $\beta$ should be chosen in the range $[0.5,2]$, with smaller values
used for networks with relatively large spectral gap.  Usings values
of $\beta$ smaller than $0.5$ results in rankings very close to those 
obtained using degree centrality,
and using $\beta > 2$ leads to rankings very close to those obtained using
eigenvector centrality. Since both degree and eigenvector centrality are
cheaper than exponential-based centrality measures, it 
would make little sense to use the matrix exponential with values 
of $\beta$ outside the interval $[0.5,2]$. As a default value, $\beta = 1$
(as originally proposed in \cite{estradarodriguez05})
is a very reasonable choice.

Similarly, resolvent-based centrality measures are most informative when
the parameter $\alpha$ is of the form $\tau/\lambda_1$ with $\tau$ chosen in the 
interval $[0.5, 0.9]$. Outside of this interval, the rankings obtained are
very close to the degree (for $\tau < 0.5$) and eigenvector (for
$\tau > 0.9$) rankings, especially when attention is restricted to the
top ranked nodes. Again, the smaller values should be used when the 
network has a large spectral gap. 

Similar conclusions hold 
for the choice of the damping parameter $\alpha$ used in the
PageRank algorithm, in broad agreement with the results of \cite{Boldi05,Boldi09}.

\section{Conclusions}
\label{sec:conclusions}
We have studied 
a broad family of parameterized network centrality measures
that includes subgraph, total communicability and Katz centrality as
well as degree and eigenvector centrality (which appear as limiting cases
of the others as the parameter approaches certain values).
Our analysis applies (for the most part) to rather general types of networks,
including directed and weighted networks; some of our results also hold
in the presence of loops. A discussion of the limiting
behavior of PageRank was also given, particularly for
small values of the parameter $\alpha$.

Our results
help explain the frequently observed correlations between the degree and
eigenvector centrality rankings on many real-world complex networks,
particulary those exhibiting a large spectral gap,
and why the rankings tend to be most stable 
precisely near the extreme values of the parameters. This is
at first sight surprising, given that
as the parameters approach their
upper bounds, the centrality scores and their derivatives diverge,
indicating extreme sensitivity.

We have discussed the role of network properties, such as the
spectral gap and the clustering coefficient, on the rate
at which the rankings obtained by a parameterized
centrality measure approach those obtained by
the degree and eigenvector centrality in the limit.
We have further shown that the parameter plays the role of 
a ``knob" that can be used to give more or less weight 
to walks of different lengths on the graph. 

In the case of resolvent and
exponential-type centrality measures, we have provided rules of
thumb for the choice of the parameters $\alpha$ and $\beta$.
In particular, we provide guidelines for the choice of
the parameters that produce rankings that
are the most different from the degree and eigenvector centrality
rankings and, therefore, most useful in terms of providing additional information
in the analysis of a given network.  Of course, the larger the
spectral gap, the smaller the range of parameter values leading to
rankings exhibiting a noticeable difference from those obtained
from degree and/or eigenvector centrality. Since degree and
eigenvector centrality are considerably less expensive to compute
compared to subgraph centrality, for
networks with large spectral gap it may be difficult to justify
the use of the more expensive centrality measures discussed
in this paper. 

Finally, in this paper we have mostly avoided discussing computational
aspects of the ranking methods under consideration, focusing 
instead on the theoretical understanding of the relationship
among the various centrality measures. For recent progress
on walk-based centrality computations see, e.g.,
\cite{benziboito,BKlym13,bonchietal,Fenu1,Fenu2}.

\section*{Acknowledgments}
The authors would like to thank Ernesto Estrada and
Shanshuang Yang for valuable discussions. 
We also thank Dianne O'Leary,   
two anonymous referees and the handling editor 
for many useful suggestions.

\pagebreak

\pagebreak

\appendix
\section{Supplementary materials to the paper}

\begin{abstract}
This document contains details of numerical experiments
performed to illustrate the theoretical results presented
in our accompanying paper.
\end{abstract}



\subsection{Limiting behavior of PageRank for small $\alpha$}
In this section we want to illustrate the behavior of the
PageRank vector in the limit of small values of the parameter $\alpha$.
We take the following example from \cite[pp.~32--33]{Pagerank_supp}.
Consider the simple digraph $G$ with $n=6$ nodes described in
Fig.~\ref{fig:1}. 

\begin{figure}[h]
\centering
\includegraphics[width=0.30\textwidth]{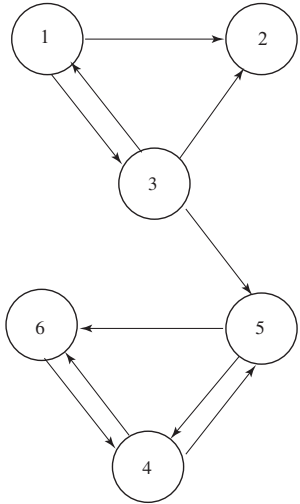}
\caption{A directed network with six nodes.}
\label{fig:1}
\end{figure} 

The adjacency matrix for this network is
$$
A =  \left(\begin{array}{cccccc}
0 &  1 &  1 & 0 & 0 & 0 \\
0 &  0 &  0 & 0 & 0 & 0 \\
1 &  1 &  0 & 0 & 1 & 0 \\
0 &  0 &  0 & 0 & 1 & 1 \\
0 &  0 &  0 & 1 & 0 & 1 \\
0 &  0 &  0 & 1 & 0 & 0 \\
\end{array}\right)\,.
$$

The corresponding matrix $H$ is obtained by transposing $A$ and
normalizing each nonzero column of $A^T$ by the sum of its 
entries:\footnote{It is worth noting that our matrices and vectors are the transposes
of those found in \cite{Pagerank_supp} since we write our proability 
distribution vectors as column vectors rather than row ones.}
$$
H =  \left(\begin{array}{cccccc}
0 &  0 &  1/3 & 0 & 0 & 0 \\
1/2 &  0 &  1/3 & 0 & 0 & 0 \\
1/2 &  0 &  0 & 0 & 0 & 0 \\
0 &  0 &  0 & 0 & 1/2 & 1 \\
0 &  0 &  1/3 & 1/2 & 0 & 0 \\
0 &  0 &  0 & 1/2 & 1/2 & 0 \\
\end{array}\right)\,.
$$
Next, we modify the second column of $H$ in order to have a
column-stochastic matrix:
$$
S =  \left(\begin{array}{cccccc}
0 &  1/6 &  1/3 & 0 & 0 & 0 \\
1/2 &  1/6 &  1/3 & 0 & 0 & 0 \\
1/2 &  1/6 &  0 & 0 & 0 & 0 \\
0 &  1/6 &  0 & 0 & 1/2 & 1 \\
0 &  1/6 &  1/3 & 1/2 & 0 & 0 \\
0 &  1/6 &  0 & 1/2 & 1/2 & 0 \\
\end{array}\right)\,.
$$
Note that $S$ is reducible. Finally, we form the matrix
$$ P = \alpha S + \frac{(1-\alpha)}{6} {\bf 1} {\bf 1}^T, \quad {\rm where}
\quad \alpha \in (0,1).$$ 
This matrix is strictly positive for any $\alpha \in (0,1)$, hence for
each such $\alpha$ there is a unique dominant eigenvector, 
the corresponding PageRank vector.
 
Now we compute the PageRank vector ${\bf p} = {\bf p}(\alpha)$
for different values of $\alpha$, and compare the corresponding
rankings of the nodes of $G$. We begin with $\alpha = 0.9$, the
value used in \cite[p.~39]{Pagerank_supp}. Rounded to five digits, the
corresponding PageRank vector is
$${\bf p}(0.9) =  \left(\begin{array}{cccccc}
.03721 & .05396 & .04151 & .37510 & .20600 & .28620 \\
\end{array}\right)^T\,.
$$
Therefore, the nodes of $G$ are ranked by their importance
as $\left(\begin{array}{cccccc} 4 & 6 & 5 & 2 & 3 & 1 
\end{array}\right)$.

Next we compute the PageRank vector for $\alpha = 0.1$:
$${\bf p}(0.1) =  \left(\begin{array}{cccccc}
.15812 & .16603 & .16067 & .17812 & .16703 & .17002 \\
\end{array}\right)^T\,.
$$
Therefore, the nodes of $G$ are ranked by their importance
as $\left(\begin{array}{cccccc} 4 & 6 & 5 & 2 & 3 & 1
\end{array}\right)$, exactly the same ranking as before.
The scores are now closer to one another (since they are all
approaching the uniform probability $1/6$), but not so close as to
make the ranking impossible, or different than in the case
of $\alpha = 0.9$.

For $\alpha = 0.01$ we find
$${\bf p}(0.01) =  \left(\begin{array}{cccccc}
 .16583 & .16666 & .16610 & .16778 & .16667 & .16695\\
\end{array}\right)^T\,.
$$
Again, we find that the nodes of $G$ are ranked 
as $\left(\begin{array}{cccccc} 4 & 6 & 5 & 2 & 3 & 1
\end{array}\right)$, exactly as before.

Finally, for $\alpha = 0.001$ we find, rounding this time
the results to seven digits:
$${\bf p}(0.001) =  \left(\begin{array}{cccccc}
.1665833 & .1666666 &  .1666111 & .1667778 & .1666667 & .1666945
\\
\end{array}\right)^T\,.
$$
As before, the ranking of the nodes is unchanged.

Clearly, as $\alpha$ gets smaller it becomes more difficult to rank the
nodes, since the corresponding PageRank values get closer and closer
together, and more accuracy is required. For this reason, it is better
to avoid tiny values of $\alpha$. This is especially true for large
graphs, where most of the individual entries of the PageRank vector
are very small.  But the important point
here is that even for very small nonzero values of $\alpha$ the  
underlying graph structure continues to influence the rankings of the 
nodes. Taking values of $\alpha$ close to 1 is probably not necessary
in practice, especially recalling that $\alpha$ values near 1 result in
slow convergence of the PageRank iteration. 

As discussed in the paper (Theorem 6.1), the rankings given by PageRank 
approach those obtained using the vector $H{\bf 1}$ (equivalently,
$S{\bf 1}$) in the limit
as $\alpha\to 0+$.  This vector is
given by
$$H{\bf 1} = \left(\begin{array}{cccccc}
1/3  &  5/6  & 1/2 & 3/2 & 5/6 & 1\\
\end{array}\right)^T\,.
$$
The corresponding ranking is again 
$\left(\begin{array}{cccccc} 4 & 6 & 5 & 2 & 3 & 1
\end{array}\right)$, with nodes 5 and 2 tied in third place. This
is in complete agreement with our analysis.
Moreover, it suggests that an inexpensive alternative to computing the
PageRank vector could be simply taking the row sums of $H$. This
of course amounts to ranking the nodes of the digraph using a
kind of weighted in-degree. 
This ranking scheme is much more crude than PageRank, as
we can see from the fact that it assigns the same score to nodes
2 and 5, whereas PageRank clearly gives higher importance to
node 5 when $\alpha=0.9$.
We make no claims about the usefulness
of this ranking scheme for real directed networks, but given its
low cost it may be worthy
of further study.

\subsection{Numerical experiments on undirected networks}
\label{sec:numerical_experiments_supp}
In this section we present the results of numerical 
experiments aimed at illustrating the limiting behavior
of walk-based, parameterized 
centrality measures using various undirected networks.
We focus our attention on exponential-type and
resolvent-type centrality measures, and study their
relation to degree and eigenvector centrality.

The rankings produced by the various centrality measures are 
compared using the {\em intersection distance} method  
(for more information, see \cite{Faginetal} and 
\cite{Boldi,Randomalpha_supp}).  Given two ranked lists $x$ and
$y$, the top-$k$ 
intersection distance is computed by: 
$${\rm isim}_k(x,y) := \frac{1}{k} 
\sum_{i=1}^k \frac{|x_i \Delta y_i|}{2i}$$ 
where $\Delta$ is the 
symmetric difference operator between the two sets and $x_k$ and 
$y_k$ are the top $k$ items in $x$ and $y$, respectively.  The top-$k$ 
intersection distance gives the average of the normalized symmetric 
differences for the lists of the top $i$ items for all $i \leq k$.  
If the ordering of the top $k$ nodes is the same for the two 
ranking schemes, ${\rm isim}_k(x,y) = 0$.  If the top $k$ are 
disjoint, then ${\rm isim}_k(x,y) = 1$.  Unless otherwise specified, 
we compare the intersection distance for the full set of ranked nodes. 


The networks come from a range of sources, although most can be found 
in the University of Florida Sparse Matrix Collection \cite{UFsparse}.  
The first is the Zachary Karate Club network, which is a classic 
example in network analysis \cite{karate}.  The Intravenous Drug User
 and the Yeast PPI networks were provided by Prof.~Ernesto Estrada 
and are not present in the University of Florida Collection.  
The three Erd\"os networks correspond to various subnetworks of 
the Erd\"os collaboration network and can be found in the Pajek 
group of the UF Collection.  The ca-GrQc and ca-HepTh networks are 
collaboration networks corresponding to the General Relativity and 
High Energy Physics Theory subsections of the arXiv and can be found 
in the SNAP group of the UF Collection.  The as-735 network can also 
be found in the SNAP group and represents the communication network 
of a group of Autonomous Systems on the Internet.  This communication 
was measured over the course of 735 days, between November 8, 1997 
and January 2, 2000.  The final network is the network of Minnesota 
roads and can be found in the Gleich group of the UF Collection.  
Basic data on these networks, including the order $n$, number of 
nonzeros, and the largest two eigenvalues, can be found in 
Table \ref{tbl:basic_data}.  All of the networks, with the exception 
of the Yeast PPI network, are simple.  The Yeast PPI network has 
several ones on the diagonal, representing the self-interaction of 
certain proteins.  All are undirected.

\begin{table}
\centering
\caption{Basic data for the networks used in the experiments.}
\begin{tabular}{|c|c|c|c|c|}
\hline
Graph & $n$ & $nnz$ & $\lambda_1$ & $\lambda_2$ \\
 \hline
 \hline
Zachary Karate Club & 34 & 156 & 6.726 & 4.977\\
 \hline
 Drug User & 616 & 4024 & 18.010 & 14.234\\
 \hline
 Yeast PPI & 2224 & 13218 & 19.486 & 16.134\\
 \hline
Pajek/Erdos971 & 472 & 2628 & 16.710 & 10.199\\
\hline
Pajek/Erdos972 & 5488 & 14170 & 14.448 & 11.886\\
\hline
Pajek/Erdos982 & 5822 & 14750 & 14.819 & 12.005\\
\hline
Pajek/Erdos992 & 6100 & 15030 & 15.131 & 12.092\\
\hline
SNAP/ca-GrQc & 5242 & 28980 & 45.617 & 38.122\\
\hline
SNAP/ca-HepTh & 9877 & 51971 & 31.035 & 23.004\\
\hline
SNAP/as-735 & 7716 & 26467 & 46.893 & 27.823\\
\hline
Gleich/Minnesota & 2642 & 6606 & 3.2324 & 3.2319\\
 \hline
\end{tabular}
\label{tbl:basic_data}
\end{table}

\subsubsection{Exponential subgraph centrality and total communicability}
\label{sec:exp_tests}
We examined the effects of changing $\beta$ on the exponential 
subgraph centrality and total communicability rankings of nodes 
in a variety of undirected real world networks, as well as their 
relation to degree and eigenvector centrality.  
Although the 
only restriction on $\beta$ is that it must be greater than 
zero, there is often an implicit upper limit that may be 
problem-dependent.  For the analysis in this section, we 
impose the following limits: $0.1 \leq \beta \leq 10$.  To examine 
the sensitivity of the exponential subgraph centrality and 
total communicability rankings, we calculate both sets of scores and 
rankings for various choices of $\beta$.  The values of $\beta$ 
tested are: 0.1, 0.5, 1, 2, 5, 8 and 10.  

The rankings produced by  the matrix exponential-based centrality 
measures for all choices of $\beta$ were compared to those produced 
by degree centrality and eigenvector centrality, using the intersection 
distance method described above.   
Plots of the intersection distances for the rankings produced by 
various choices of $\beta$ with those produced by degree or eigenvector 
centrality can be found in Figs.~\ref{fig:real_exp_deg} and 
\ref{fig:real_exp_eig}.  The intersection distances for rankings 
produced by successive choices of $\beta$ can be found in 
Fig.~\ref{fig:real_exp}.

\begin{figure}[t!]
\centering
\includegraphics[width=1\textwidth]{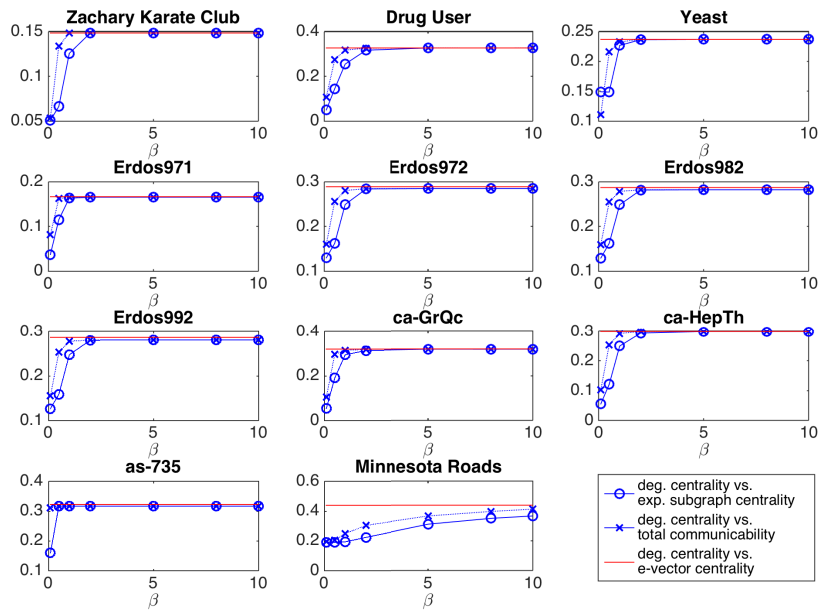}
\caption{The intersection distances between degree centrality 
and the exponential subgraph centrality (blue circles) or total 
communicability (blue crosses) rankings of the nodes in the networks 
in Table \ref{tbl:basic_data}. The red lines are the intersection distances between
degree centrality and eigenvector centrality and are added for reference.}
\label{fig:real_exp_deg}
\end{figure} 

In Figure \ref{fig:real_exp_deg}, the rankings produced by exponential 
subgraph centrality and total communicability are compared to those 
produced by degree centrality.  For small values of $\beta$, both 
sets of rankings based on the matrix exponential are very close to 
those produced by degree centrality (low intersection distances).  
When $\beta =0.1$, the largest intersection distance between the 
degree centrality rankings and the exponential subgraph centrality 
rankings for the networks examined is slightly less than 0.2 (for 
the Minnesota road network).  The largest intersection distance 
between the total communicability rankings with $\beta = 0.1$ and 
the degree centrality rankings is 0.3 (for the as-735 network).  
In general, the (diagonal-based) exponential subgraph centrality 
rankings tend to be slightly closer to the degree rankings than 
the (row sum-based) total communicability rankings for low values 
of $\beta$.  As $\beta$ increases, the intersection distances increase, 
then level off.  The rankings of nodes in networks with a very large 
(relative) spectral gap, such as the karate, Erdos971 and as-735 
networks, stabilize extremely quickly, as expected.  The one exception 
to the stabilization is the intersection distances between the degree 
centrality rankings and exponential subgraph centrality (and total 
communicability rankings) of nodes in the Minnesota road network.  
This is also expected, as the tiny ($<0.001$) spectral gap for the Minnesota road 
network means that it will take longer for the exponential subgraph 
centrality (and total communicability) rankings to stabilize as $\beta$ 
increases. It is worth noting that the Minnesota road network is 
quite different from the other ones: it is (nearly) planar, has
large diameter and a much more regular degree distribution. 

\begin{figure}[t!]
\centering
\includegraphics[width=1\textwidth]{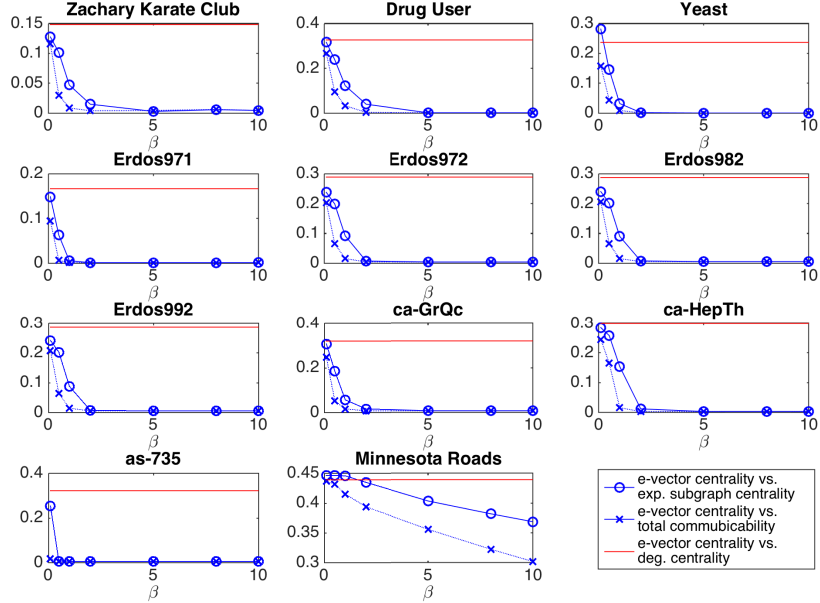}
\caption{The intersection distances between eigenvector centrality and the exponential 
subgraph centrality (blue circles) or total communicability (blue crosses) rankings of the 
nodes in the networks in Table \ref{tbl:basic_data}.  The red lines show the intersection 
distance between eigenvector centrality and degree centrality.}
\label{fig:real_exp_eig}
\end{figure} 

The rankings produced by exponential subgraph centrality and total 
communicability are compared to those produced by eigenvector 
centrality for various values of $\beta$ in Figure \ref{fig:real_exp_eig}.  
When $\beta$ is small, the intersection distances are large but, as 
$\beta$ increases, the intersection distances quickly decrease.  
When $\beta =2$, they are essentially zero for all but one of the 
networks examined.  Again, the outlier is the Minnesota road network.  
For this network, the intersection distances between the 
exponential-based centrality rankings and the eigenvector centrality 
rankings still decrease as $\beta$ increases, but at a much slower 
rate than for the other networks.  This is also expected, inview 
of the very small spectral gap.
Again, the rankings of the nodes in the karate, Erdos971, 
and as-735 networks, which have very large relative spectral gaps, 
stabilize extremely quickly.

\begin{figure}[t!]
\centering
\includegraphics[width=0.48\textwidth]{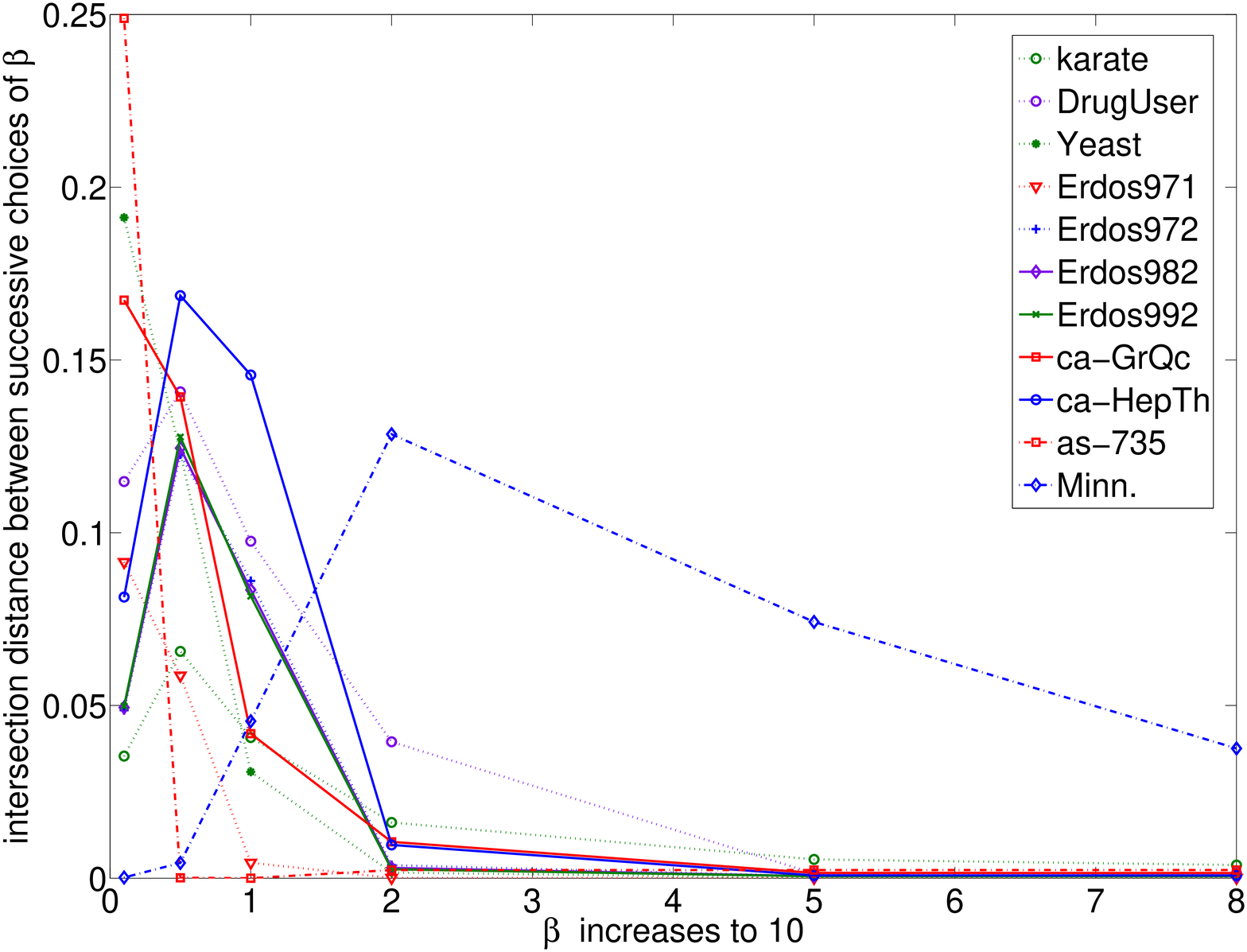}
\hspace{0.1in}
\includegraphics[width=0.48\textwidth]{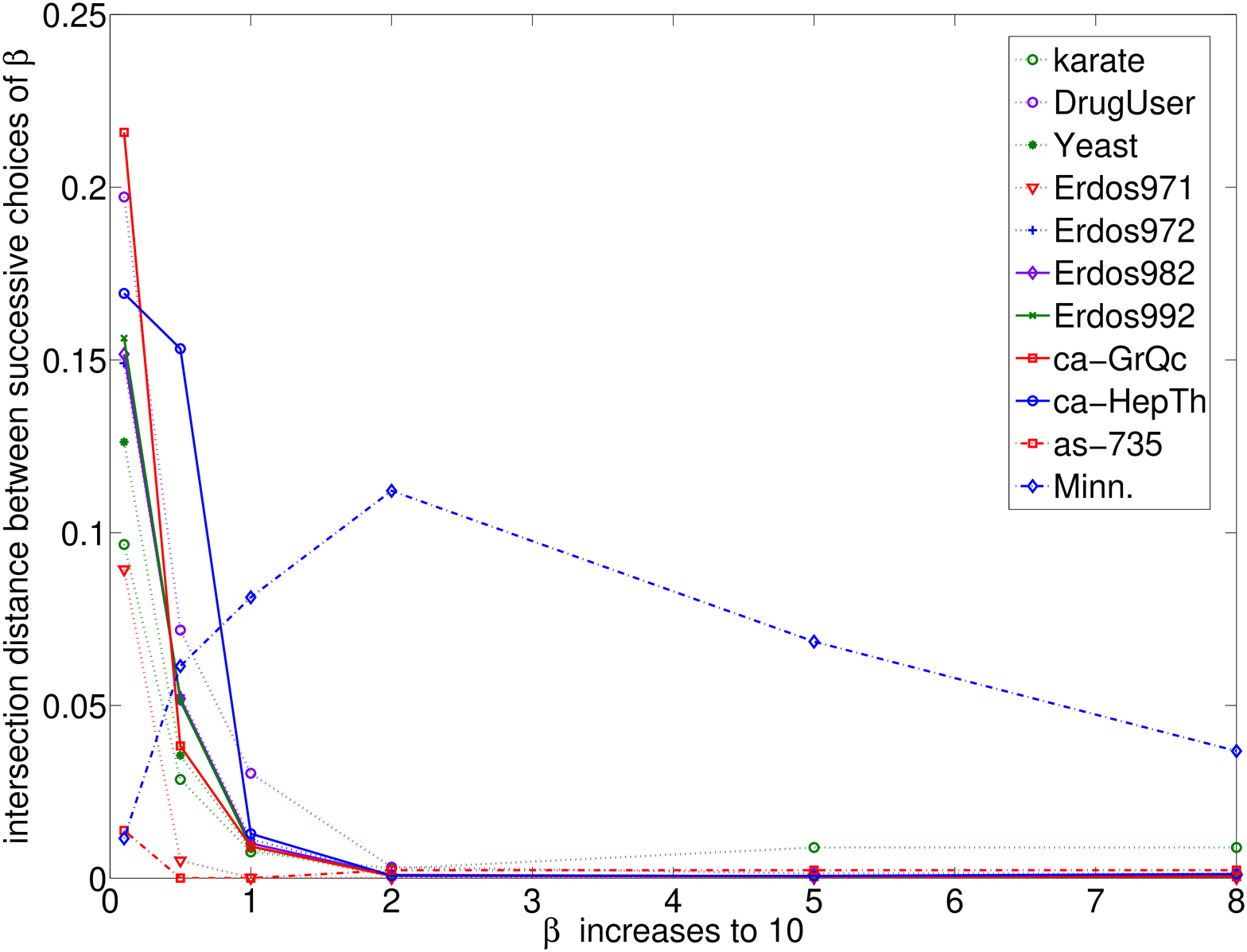}
\caption{The intersection distances between the exponential 
subgraph centrality (left) or total communicability (right) 
rankings produced by successive choices of $\beta$. Each 
line corresponds to a network in Table \ref{tbl:basic_data}.}
\label{fig:real_exp}
\end{figure} 

In Figure \ref{fig:real_exp}, the intersection distances between the 
rankings produced by exponential subgraph centrality and total 
communicability are compared for successive choices of $\beta$.  
Overall, these intersection distances are quite low (the highest is 
0.25 and occurs for the exponential subgraph centrality rankings of 
the as-735 network when $\beta$ increases from 0.1 to 0.5).  For all 
the networks examined, the largest intersection distances between 
successive choices of $\beta$ occur as $\beta$ increases to two.  
For higher values of $\beta$, the intersection distance drops, which 
corresponds to the fact that the rankings are converging to those 
produced by eigenvector centrality.  In general, there is less change 
in the rankings produced by the total communicability scores for 
successive values of $\beta$ than for the rankings produced by the 
exponential subgraph centrality scores.

If the intersection distances are restricted to the top 10 nodes, 
they are even lower.  For the karate, Erdos992, and ca-GrQc networks, 
the intersection distance for the top 10 nodes between successive 
choices of $\beta$ is always less than 0.1.  For the DrugUser, 
Yeast, Erdos971, Erdos982, and ca-HepTh networks, the intersection 
distances are somewhat higher for low values of $\beta$, but by the 
time $\beta =2$, they are all equal to 0 as the rankings have converged 
to those produced by the eigenvector centrality.  For the Erdos972 
network, this occurs slightly more slowly.  The intersection distances 
between the rankings of the top 10 nodes produced by $\beta =2$ and 
$\beta =5$ are 0.033 and for all subsequence choices of $\beta$ are 0.  
In the case of the Minnesota Road network, the intersection distances 
between the top 10 ranked nodes never stabilize to 0, as is expected.
More detailed results and plots can be found in \cite[Appendix B]{Klymko13_supp}.

For the networks examined, when $\beta < 0.5$, the exponential subgraph 
centrality and total communicability rankings are very close to those 
produced by degree centrality.  When $\beta \geq 2$, they are essentially 
identical to the rankings produced by eigenvector centrality.  Thus, 
the most additional information about node rankings (i.e. information 
that is not contained in the degree or eigenvector centrality rankings) 
is obtained when $0.5 < \beta < 2$.  This supports the intuition 
developed in section 5 of the of the accompanying paper that ``moderate" values of 
$\beta$ should be used to gain the most benefit from the use of matrix 
exponential-based centrality rankings.

\subsubsection{Resolvent subgraph and Katz centrality}
\label{sec:Katz_tests}
In this section we investigate the effect of changes in $\alpha$ 
on the resolvent subgraph centrality and Katz centrality in the 
networks listed in Table \ref{tbl:basic_data}, as well as the relationship 
of these centrality measures to degree and eigenvector centrality.  
We calculate the scores and node rankings produced by degree and
eigenvector centrality, as well as those produced by the
resolvent ($RC_i(\alpha)$) and Katz 
($K_i(\alpha)$) centralities for various values of $\alpha$.  The values 
of $\alpha$ tested are given by $\alpha = 0.01\cdot\frac{1}{\lambda_1}$, 
$0.05\cdot\frac{1}{\lambda_1}$, $0.1\cdot\frac{1}{\lambda_1}$, 
$0.25\cdot\frac{1}{\lambda_1}$, $0.5\cdot\frac{1}{\lambda_1}$, 
$0.75\cdot\frac{1}{\lambda_1}$, $0.9\cdot\frac{1}{\lambda_1}$, 
$0.95\cdot\frac{1}{\lambda_1},$ and $0.99\cdot\frac{1}{\lambda_1}$. 

As in section \ref{sec:exp_tests}, the rankings produced by degree 
centrality and eigenvector centrality were compared to those produced 
by resolvent-based centrality measures for all choices of 
$\alpha$ using the intersection distance method.  The results are 
plotted in Figs.~\ref{fig:real_res_deg} and \ref{fig:real_res_eig}.  
The rankings produced by successive choices of $\alpha$ are also 
compared and these intersection distances are plotted in 
Fig.~\ref{fig:real_res}.

\begin{figure}[t!]
\centering
\includegraphics[width=1\textwidth]{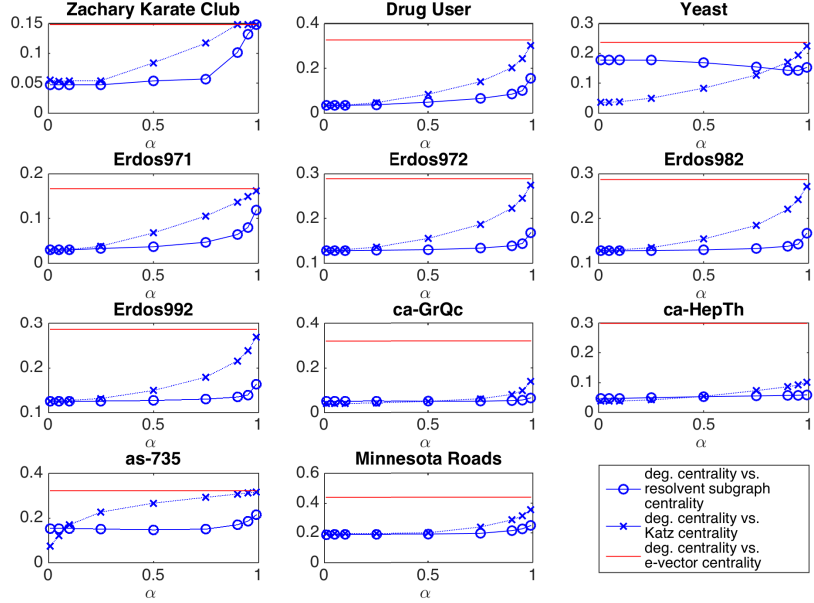}
\caption{The intersection distances between degree centrality and 
the resolvent subgraph centrality (blue circles) or Katz centrality (blue crosses) 
rankings of the nodes in the networks in Table \ref{tbl:basic_data}.  The $x$-axis 
measures $\alpha$ as a percentage of its upper bound, $\frac{1}{\lambda_1}$.  The red lines
show the intersection distances between degree centrality and eigenvector centrality
for each of the networks.}
\label{fig:real_res_deg}
\end{figure} 

Fig.~\ref{fig:real_res_deg} shows the intersection distances between 
the degree centrality rankings and those produced by resolvent subgraph 
centrality or Katz centrality for the values of $\alpha$ tested. 
When $\alpha$ is small, the intersection distances between the 
resolvent-based centrality rankings and the degree centrality 
rankings are low.  For $\alpha = 0.01\cdot\frac{1}{\lambda_1}$, the 
largest intersection distance between the degree centrality rankings 
and the resolvent subgraph centrality rankings is slightly less than 
0.2 (for the Minnesota road network). The largest intersection distance 
between the degree centrality rankings and the Katz centrality rankings 
is also slightly less than 0.2 (again, for the Minnesota road network).  
The relatively large intersection distances for the node rankings on 
the Minnesota road network when $\alpha = 0.01\cdot\frac{1}{\lambda_1}$ 
is due to the fact that with both the degree centrality and the resolvent 
subgraph (or Katz) centrality, there are many nodes with very close scores.
Thus, small changes in the score values (induced by small changes in $\alpha$)
can lead to large changes 
in the rankings.  As $\alpha$ increases towards $\frac{1}{\lambda_1}$, 
the intersection distances increase.  This increase is more rapid 
for the Katz centrality rankings than for the resolvent subgraph 
centrality rankings.

\begin{figure}[t!]
\centering
\includegraphics[width=1\textwidth]{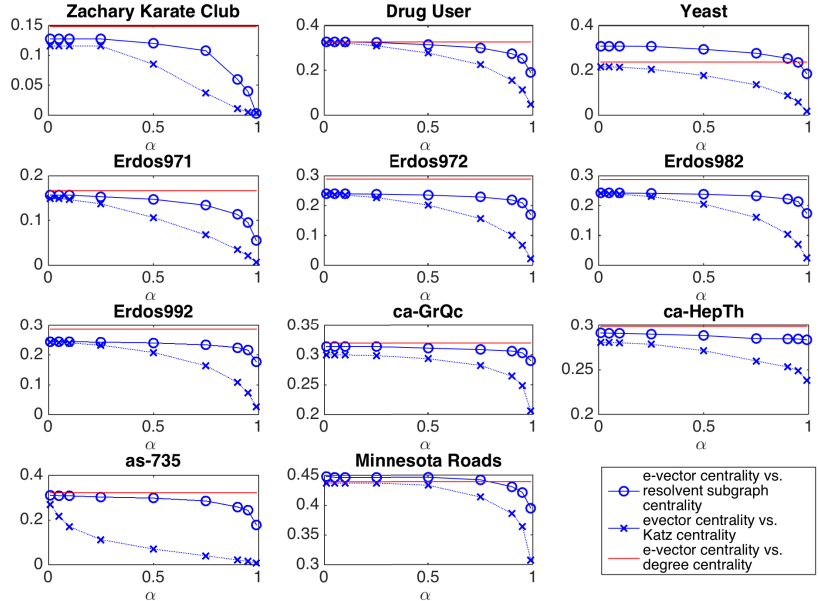}
\caption{The intersection distances between eigenvector centrality 
and the resolvent subgraph centrality (blue circles) or Katz centrality (blue crosses) 
rankings of the nodes in the networks in Table \ref{tbl:basic_data}.  The $x$-axis
measures $\alpha$ as a percentage of its upper bound, $\frac{1}{\lambda_1}$. The 
red reference lines show the intersection distance between eigenvector centrality and 
degree centrality.}
\label{fig:real_res_eig}
\end{figure} 

In Fig.~\ref{fig:real_res_eig}, the resolvent subgraph centrality 
and Katz centrality rankings for various values of $\alpha$ are 
compared to the eigenvector centrality rankings on the networks 
described in Table \ref{tbl:basic_data}.  For small values of $\alpha$, 
the intersection distances tend to be large.  As $\alpha$ increases, 
the intersection distances decrease for both resolvent subgraph 
centrality and Katz centrality on all of the networks examined.  
This decrease is faster for the (row sum-based) Katz centrality 
rankings than for the (diagonal-based) resolvent subgraph centrality 
rankings.  The network with the highest intersection distances between 
the eigenvector centrality rankings and those based on the matrix 
resolvent, and slowest decrease of these intersection distances as 
$\alpha$ increases, is the Minnesota road network.  As was the case 
when matrix exponential-based scores were examined, this is expected 
due to this network's small spectral gap.  The node rankings in 
networks with large relative spectral gaps (karate, Erdos971, 
as-735) converge to the eigenvector centrality rankings most quickly.

\begin{figure}[t!]
\centering
\includegraphics[width=0.48\textwidth]{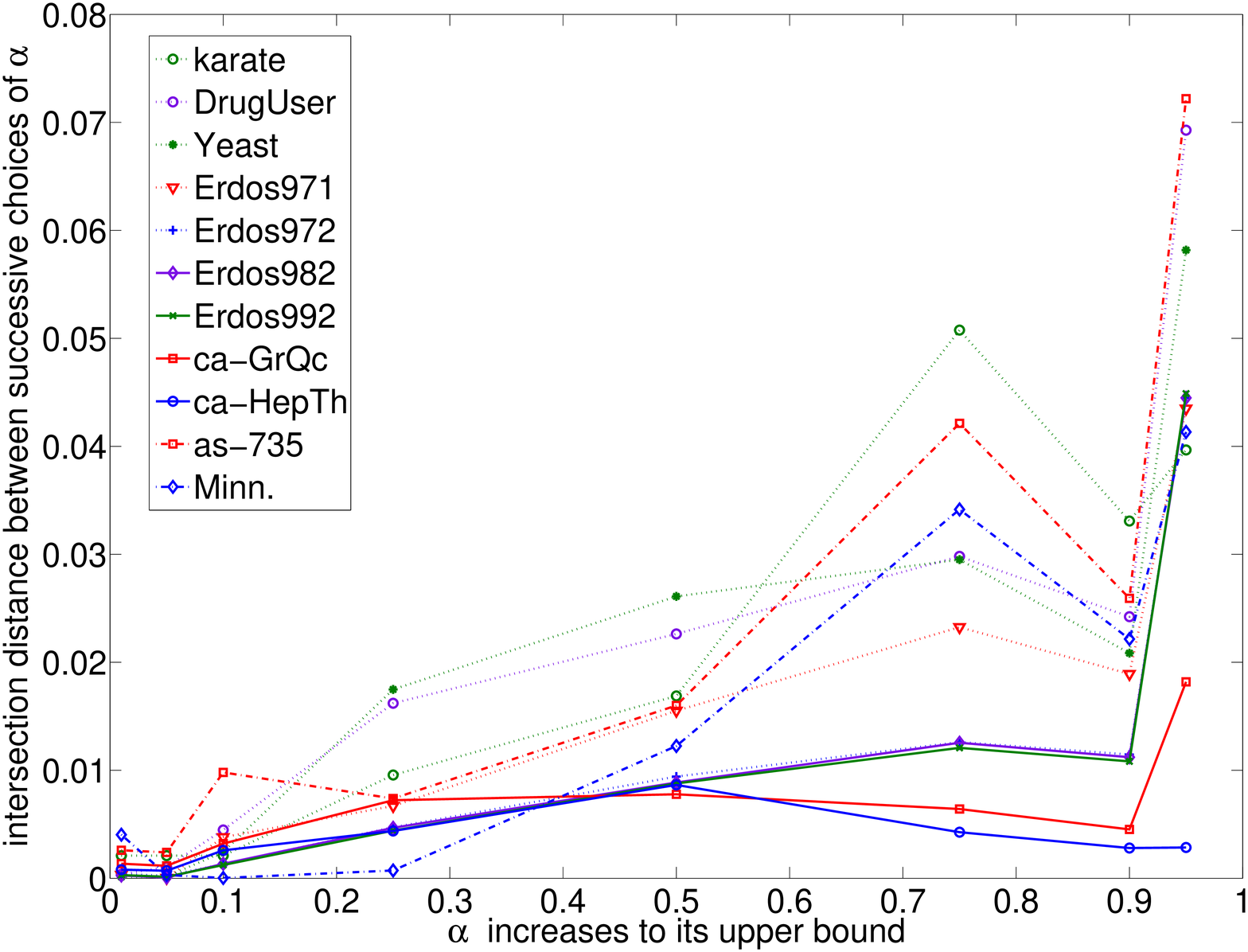}
\hspace{0.1in}
\includegraphics[width=0.48\textwidth]{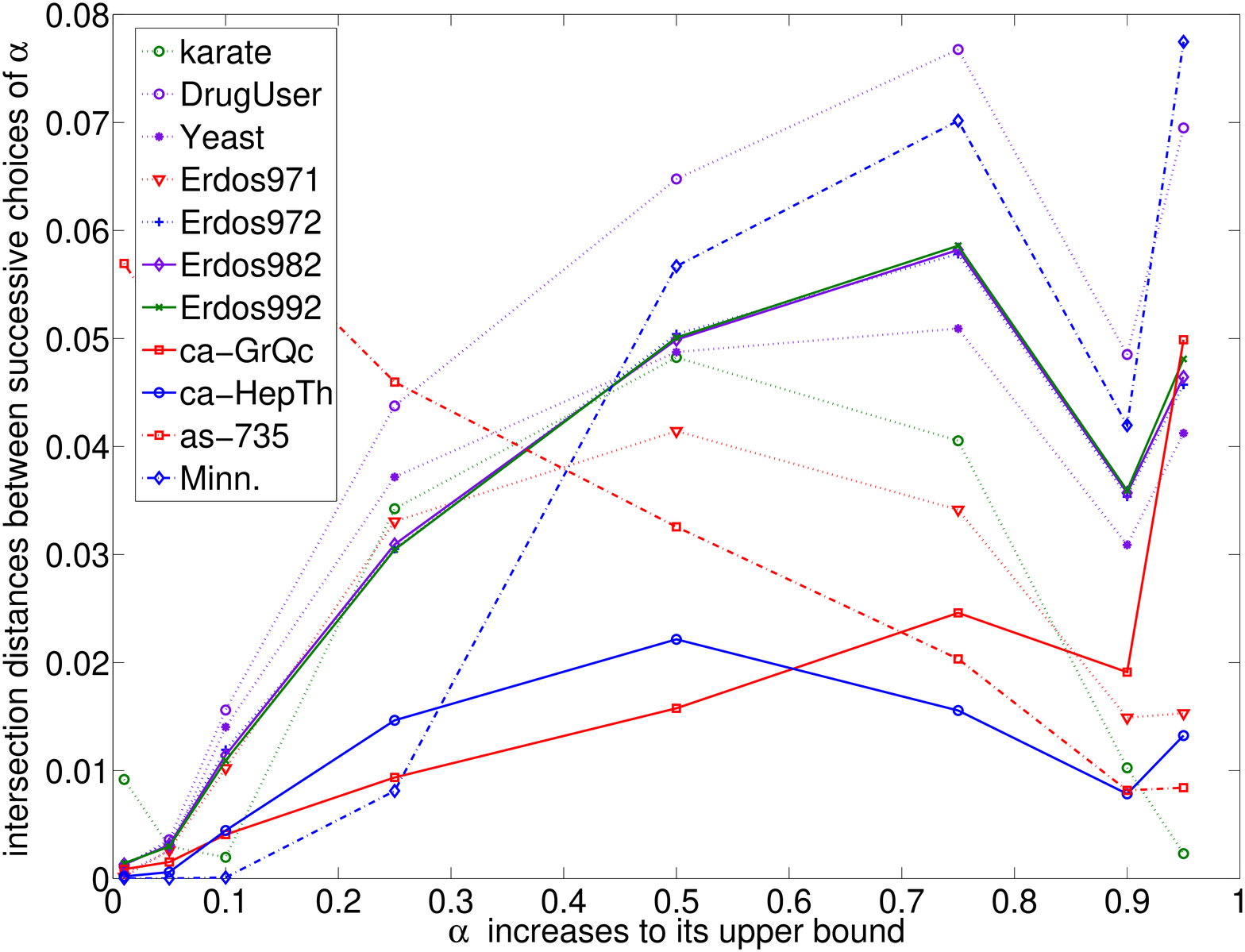}
\caption{The intersection distances between resolvent subgraph 
centrality (left) or Katz centrality (right) rankings produced 
by successive choices of $\alpha$.  Each line corresponds to a
network in Table \ref{tbl:basic_data}.}
\label{fig:real_res}
\end{figure}

The intersection distance between the resolvent subgraph and Katz 
centrality rankings produced by successive choices of $\alpha$ are 
plotted in Fig.~\ref{fig:real_res}.  All of these intersection 
distances are extremely small (the largest is $<0.08$), indicating 
that the rankings do not change much as $\alpha$ increases.  However, 
as $\alpha$ increases, the rankings corresponding to successive 
values of $\alpha$ tend to be slightly less similar to each other.  
The  exception to this is the Katz centrality rankings for the as-735 
network which become more similar as $\alpha$ increases.

Again, if the analysis is restricted to the top 10 nodes, the 
intersection distances between the rankings produced by successive 
choices of $\alpha$ are very small.  For the karate, Erdos971, 
Erdos982, Erdos992, ca-GrQc, and Minnesota road networks, the 
intersection distances between the top 10 ranked nodes for successive 
choices of $\alpha$ are always less than or equal to 0.1 and often 
equal to zero.  For the ca-HepTh network, the top 10 ranked nodes 
are exactly the same for all choices of $\alpha$.  For the DrugUser, 
Yeast, and Erdos972 networks, they are always less than 0.2.
Detailed results can be found in \cite{Klymko13_supp}.

For the eleven networks examined, the resolvent subgraph and Katz 
centrality rankings tend to be close to the degree centrality 
rankings when $\alpha < 0.5\cdot\frac{1}{\lambda_1}$.  
It is interesting to note that as $\alpha$ increases, these rankings 
stay close to the degree centrality rankings until $\alpha$ is 
approximately one half of its upper bound.  Additionally, the  
resolvent based rankings are close to the eigenvector centrality 
rankings when $\alpha > 0.9\cdot\frac{1}{\lambda_1}$.  Thus, the 
most information is gained by using resolvent based centrality 
measures when $0.5 \cdot \frac{1}{\lambda1} \leq \alpha \leq 
0.9\cdot\frac{1}{\lambda_1}$. This supports the intuition from 
section 5 of the accompanying paper that ``moderate" values of $\alpha$ 
provide the most additional information about node ranking beyond 
that provided by degree and eigenvector centrality.   

It is worth noting that similar conclusions have been obtained 
for the choice of the damping parameter $\alpha$ used in the
PageRank algorithm; see \cite{Boldi05_supp,Boldi09_supp}.  

\subsection{Numerical experiments on directed networks}
\label{sec:directed_exp}
In this section, we examine the relationship between the 
exponential and resolvent-based broadcast centrality measures with 
the out-degrees and the dominant right eigenvectors of two real world 
directed networks.  A similar analysis can be done on the relationship 
between the receive centrality measures and the in-degrees and dominant 
left eigenvectors. For the experiments we use two networks from
the University of 
Florida Sparse Matrix Collection \cite{UFsparse}.  As before, 
the rankings are compared 
using the intersection distance method.  The first network we examine 
is wb-cs-Stanford, a network of hyperlinks between the Stanford CS 
webpages in 2001.  It is in the Gleich group of the UF collection.  
The second network is the wiki-Vote network, which is a network of 
who votes for whom in elections for Wikipedia editors to become 
administrators.  It is in the SNAP group of the UF collection. 

Since our theory applies to
strongly connected networks with irreducible adjacency matrices, our 
experiments were performed on the largest strongly connected component 
of the above networks.  Basic data on these strongly connected 
components can be found in Table \ref{tbl:basic_data_dir}.  In both 
of the networks examined, the two largest eigenvalues of the largest 
strongly connected component are real. 
Both networks are simple.

\begin{table}
\centering
\caption{Basic data on the largest strongly connected component 
of the real-world directed networks examined.}
\begin{tabular}{|c|c|c|c|c|}
\hline
Graph & $n$ & $nnz$ & $\lambda_1$ & $\lambda_2$ \\
 \hline
 \hline
Gleich/wb-cs-Stanford & 2759 & 13895 & 35.618 & 12.201 \\
 \hline
 SNAP/wiki-Vote & 1300 & 39456 & 45.145 & 27.573 \\
\hline
\end{tabular}
\label{tbl:basic_data_dir}
\end{table}

\subsubsection{Total communicability}
\label{sec:dir_exp_tests}
As in section \ref{sec:exp_tests}, we examine the effect of 
changing $\beta$ on the broadcast total communicability rankings 
of nodes in the networks, as well as their 
relation to the rankings obtained using the
out-degrees and dominant right eigenvectors of the 
networks. The measures were calculated for the networks described 
in Table \ref{tbl:basic_data_dir}.   To examine the sensitivity of 
the broadcast total communicability rankings, we calculate the scores 
and rankings for various choices of $\beta$.  The values of $\beta$ 
tested are: 0.1, 0.5, 1, 2, 5, 8 and 10.  

The broadcast rankings produced by total communicability for all 
choices of $\beta$ were compared to those produced by the out-degree 
rankings and the rankings produced by ${\bf x}_1$ using the intersection 
distance method as described in section \ref{sec:numerical_experiments_supp}.  
Plots of the intersection distances for the rankings produced by 
various choices of $\beta$ with those produced by the out-degrees 
and right dominant eigenvector can be found in 
Figs.~\ref{fig:real_exp_deg_dir} and \ref{fig:real_exp_eig_dir}. 
The intersection distances for rankings produced by successive 
choices of $\beta$ can be found in Figure \ref{fig:real_exp_dir}.

\begin{figure}[t!]
\centering
\includegraphics[width=1\textwidth]{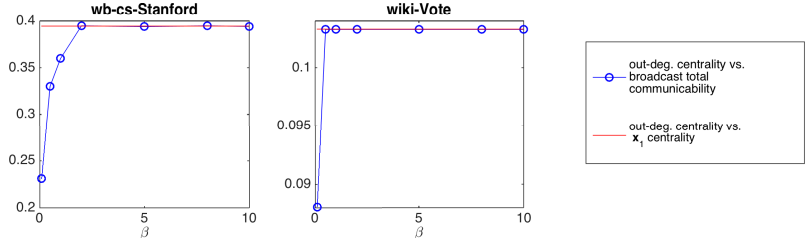}
\caption{The intersection distances between the out-degree rankings 
and the broadcast total communicability rankings (blue circles) of the nodes in 
the networks in Table \ref{tbl:basic_data_dir}.  The red reference line shows the
intersection distance between the out-degree centrality and the ${\bf x}_1$ centrality
rankings.}
\label{fig:real_exp_deg_dir}
\end{figure} 

In Fig.~\ref{fig:real_exp_deg_dir}, the intersection distances 
between the rankings produced by broadcast total communicability 
are compared to those produced by the out-degrees of nodes in the 
network.  As $\beta$ approaches 0, the intersection distances 
decrease for both networks.  As $\beta$ increases to 10, the 
intersection distances initially increase, then stabilize as the 
rankings converge to those produced by ${\bf x}_1$.

\begin{figure}[t!]
\centering
\includegraphics[width=1\textwidth]{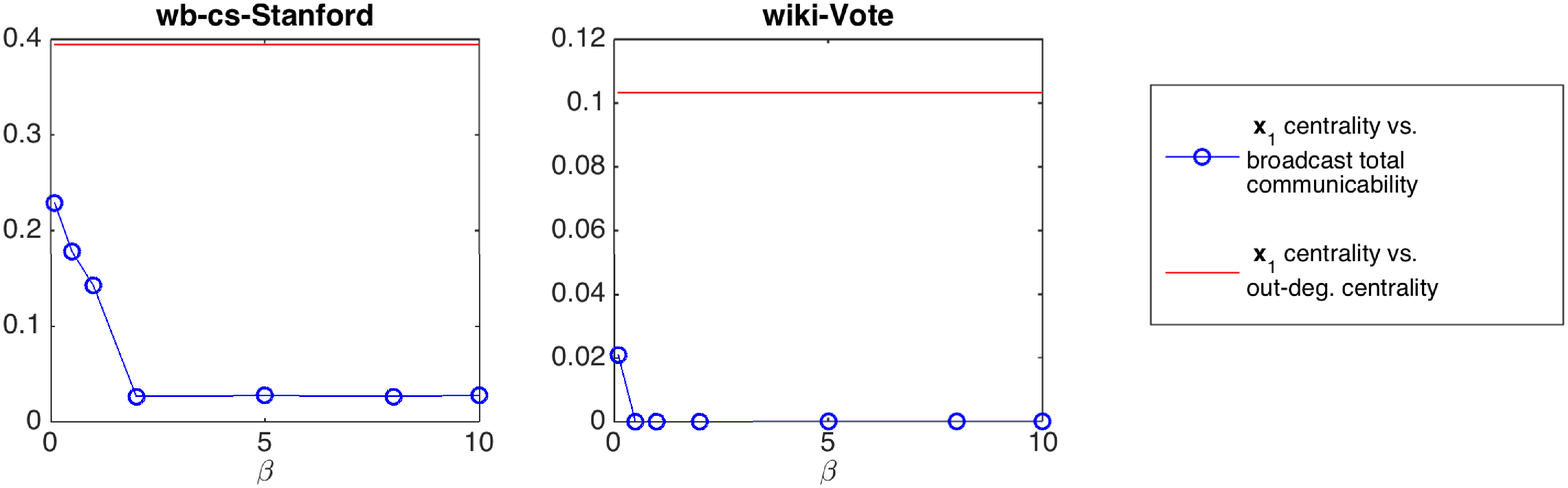}
\caption{The intersection distances (blue circles) between the rankings produced by 
 ${\bf x}_1$ and the broadcast total communicability  rankings of the nodes in the 
 networks in Table \ref{tbl:basic_data_dir}. The red lines show the intersection distances between
 the ${\bf x}_1$ rankings and those produced by the out-degrees.}
\label{fig:real_exp_eig_dir}
\end{figure} 

The intersection distances between the rankings produced by 
broadcast total communicability are compared to those produced by  
${\bf x}_1$ in Figure \ref{fig:real_exp_eig_dir}.  For both networks, 
the intersection distances quickly decrease as $\beta$ increases. 
In the wiki-Vote network, the intersection distances between the 
compared rankings are 0 by the time $\beta = 0.5$.  For the 
wb-cs-Stanford network, by the time $\beta$ has reached five, the 
intersection distances between the broadcast total communicability 
rankings and those produced by ${\bf x}_1$ have decreased to about 
0.04.  The rankings then stabilize at this intersection distance.  
This is due to a group of nodes that have nearly identical total 
communicability scores. 

\begin{figure}[t!]
\centering

\includegraphics[width=1\textwidth]{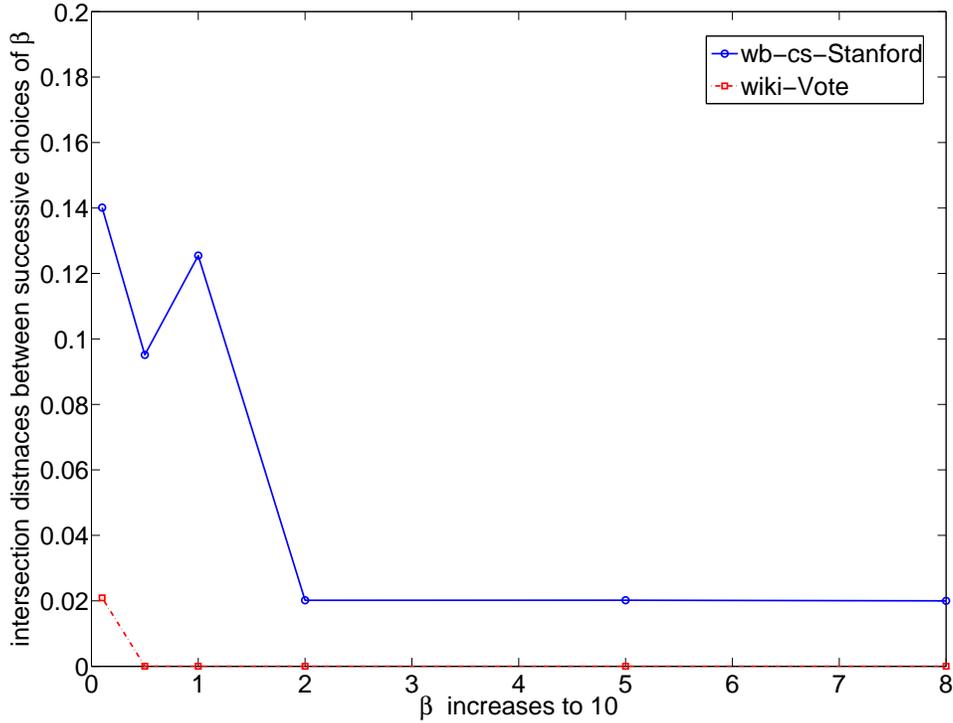}
\caption{The intersection distances between the broadcast total 
communicability rankings produced by successive choices of $\beta$. 
Each line corresponds to a network in Table \ref{tbl:basic_data_dir}.}
\label{fig:real_exp_dir}
\end{figure} 

In Fig.~\ref{fig:real_exp_dir}, the intersection distances 
between the broadcast total communicability rankings for successive 
choices of $\beta$ are plotted.  These intersection distances are 
slightly lower than those observed in the undirected case, with a 
maximum of approximately 0.14, which occurs in the wb-cs-Stanford 
network when $\beta$ increases from 0.01 to 0.05.  By the time
$\beta=0.5$, therankings on the wiki-Vote network have 
stabilized and all subsequent intersection distances are 0.  
For both the broadcast total communicability rankings on the 
wb-cs-Stanford network, the intersection distances decrease 
(non-monotonically) as $\beta$ increases until they stabilize 
at approximately 0.02.

When this analysis is restricted to the top 10 nodes, the 
intersection distances are extremely small.  For the 
wb-cs-Stanford network, the largest intersection distance between 
the top 10 ranked nodes for successive choices of $\beta$ is 0.11  
(when $\beta$ increases from 0.1 to 0.5).  For the wiki-Vote network, 
the intersection distance between the top 10 total communicability 
scores is 0.01 when $\beta$ increases from 0.1 to 0.5, and zero otherwise; 
see \cite[Appendix B]{Klymko13_supp} for detailed results and plots.

The differences between the out-degree rankings and the broadcast 
total communicability rankings are greatest when $\beta \geq 0.5$.  
The differences between the left and right eigenvector based rankings 
and the broadcast rankings are greatest when $\beta < 2$ (although 
in the case of the wiki-Vote network, they have converged by the time 
$\beta = 0.5$).  Thus, like in the case of the undirected networks, 
moderate values of $\beta$ give the most additional ranking information 
beyond that provided by the out-degrees and the left and right eigenvalues.

\subsubsection{Katz centrality}
In this section, we investigate the effect of changes in $\alpha$ 
on the broadcast Katz centrality rankings of nodes in the networks 
listed in Table \ref{tbl:basic_data_dir} and relationship of these 
centrality measures to the rankings produced by the out-degrees and 
the dominant right eigenvectors of the network.  We calculate the 
scores and node rankings produced by $K_i^b(\alpha)$ for various 
values of $\alpha$.  The values of $\alpha$ tested are given by 
$\alpha = 0.01\cdot\frac{1}{\lambda_1}$, $0.05\cdot\frac{1}{\lambda_1}$, 
$0.1\cdot\frac{1}{\lambda_1}$, $0.25\cdot\frac{1}{\lambda_1}$, 
$0.5\cdot\frac{1}{\lambda_1}$, $0.75\cdot\frac{1}{\lambda_1}$, 
$0.9\cdot\frac{1}{\lambda_1}$, $0.95\cdot\frac{1}{\lambda_1}$, 
and $0.99\cdot\frac{1}{\lambda_1}$. 

The rankings produced by the out-degrees and the dominant right 
eigenvectors were compared to those produced by Katz centrality for 
all choices of $\alpha$ using the intersection distance method, as 
was done in Section \ref{sec:dir_exp_tests}.  The results are plotted 
in Figs.~\ref{fig:real_res_deg_dir} and \ref{fig:real_res_eig_dir}.  

\begin{figure}[t!]
\centering
\includegraphics[width=1\textwidth]{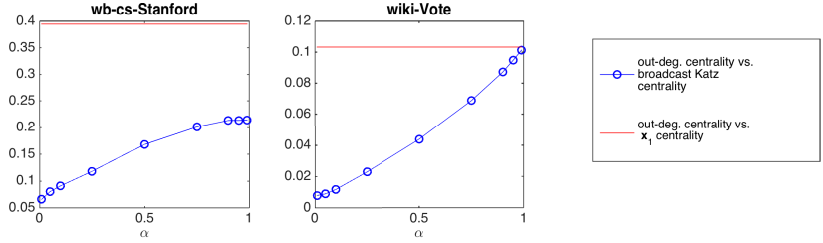}
\caption{The intersection distances (blue circles) between the rankings 
produced by the out-degrees and the broadcast 
Katz centrality rankings of the nodes in the networks in 
Table \ref{tbl:basic_data_dir}.  Here, the $x$-axis shows $\alpha$ as a
percentage of its upper bound, $\frac{1}{\lambda_1}$.  The red reference lines
show the intersection distance between the rankings produced by ${\bf x}_1$ and
those produced by the out-degrees of the nodes.}
\label{fig:real_res_deg_dir}
\end{figure} 

As $\alpha$ increases from $0.01\cdot\frac{1}{\lambda_1}$ to 
$0.99\cdot\frac{1}{\lambda_1}$, the intersection distances between 
the scores produced by the broadcast Katz centralities and the 
out-degrees increase.  When $\alpha$ is small, the broadcast Katz 
centrality rankings are very close to those produced by the out-degrees 
(low intersection distances).  On the wb-cs-Stanford network, when 
$\alpha = 0.01\cdot\frac{1}{\lambda_1}$, the intersection distance 
between the two rankings is approximately 0.06.  On the wiki-Vote 
network, it is approximately 0.01.    As $\alpha$ increases, 
the intersection distances also increase.  By the time 
$\alpha  = 0.99 \cdot \frac{1}{\lambda_1}$, the intersection distance 
between the two sets of node rankings on the wb-cs-Stanford network 
is above 0.2 and on the wiki-Vote network it is approximately 0.1.

\begin{figure}[t!]
\centering
\includegraphics[width=1\textwidth]{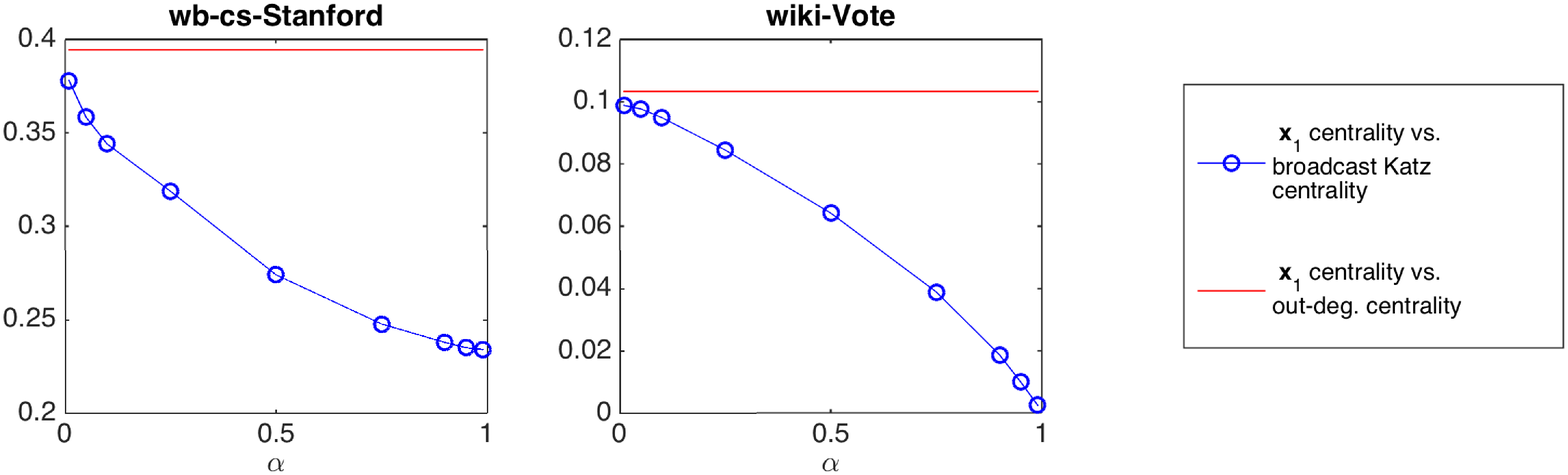}
\caption{The intersection distances (blue circles) between the rankings produced 
by ${\bf x}_1$ and the broadcast Katz centrality rankings 
of the nodes in the networks in Table \ref{tbl:basic_data_dir}.  Here, the $x$-axis shows
$\alpha$ as a percentage of its upper bound, $\frac{1}{\lambda_1}$.  The red lines show the
intersection distance between the rankings produced using ${\bf x}_1$ and the node out-degrees.}
\label{fig:real_res_eig_dir}
\end{figure} 

In Fig.~\ref{fig:real_res_eig_dir}, the rankings produced by 
broadcast Katz centrality are compared to those produced 
by ${\bf x}_1$.  Overall,  The intersection distances between 
the two sets of rankings are lower on the wiki-Vote network 
than they are on the wb-cs-Stanford network.  As $\alpha$ increases 
from $0.01\cdot\frac{1}{\lambda_1}$ to $0.99\cdot\frac{1}{\lambda_1}$, 
the intersection distances between the two sets of rankings on 
the wiki-Vote network decrease from 0.1 to essentially 0.  
On the wb-cs-Stanford network, they decrease from approximately 
0.47 to 0.24.

\begin{figure}[t!]
\centering
\includegraphics[width=0.75\textwidth]{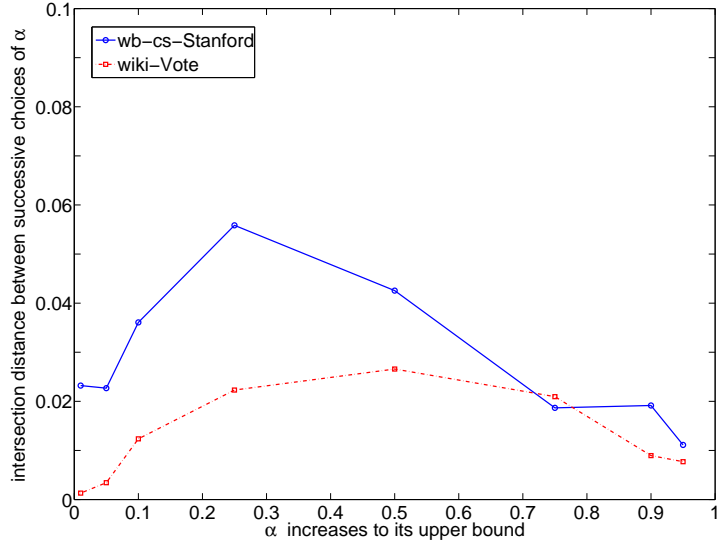}
\caption{The intersection distances between the broadcast 
Katz centrality rankings produced by successive choices of 
$\alpha$. Each line corresponds to a network in Table 
\ref{tbl:basic_data_dir}.}
\label{fig:real_res_dir}
\end{figure} 

The intersection distances between the rankings produced by the 
broadcast Katz centralities for successive values of $\alpha$ are 
plotted in Figure \ref{fig:real_res_dir}.  As was the case for the 
undirected networks, these rankings are more stable in 
regards to the choice of $\alpha$ than the total communicability 
rankings were in regards to the choice of $\beta$.  Here, the maximum 
intersection distance is less than 0.1.  When only the top 10 ranked 
nodes are considered, the intersection distances have a maximum of 0.06 
(on the wb-cs-Stanford network when $\alpha$ increases from 
$0.25\cdot\frac{1}{\lambda_1}$ to $0.5\cdot\frac{1}{\lambda_1}$).  
For both networks, the intersection distances between the rankings on 
the top 10 nodes for successive choices of $\alpha$ are quite small 
(the maximum is 0.18 and the majority are $< 0.1$).

The broadcast Katz centrality rankings are only far from those produced 
by the out-degrees when $\alpha \geq 0.5\cdot\frac{1}{\lambda_1}$.  
They are farthest from those produced by the dominant right eigenvector 
of $A$ when $\alpha < 0.9\cdot\frac{1}{\lambda_1}$.  Thus, as was 
seen in the case of undirected networks, the most additional information 
is gained when moderate values of $\alpha$, $ 0.5\cdot\frac{1}{\lambda_1} 
\leq \alpha <  0.9\cdot\frac{1}{\lambda_1}$, are used to 
calculate the matrix resolvent based centrality scores.  


\vspace{0.1in}

\subsection*{Acknowledgments}
Thanks to Carl Meyer for allowing us to use Fig.~\ref{fig:1} From \cite{Pagerank_supp}.

\end{document}